\documentclass{amsart}

\pdfoutput=1

\usepackage[hmarginratio={1:1}]{geometry}
\usepackage[shortlabels]{enumitem}
\usepackage[hidelinks]{hyperref}
\usepackage[style=alphabetic,
doi=false,
isbn=false,
maxnames=99,
maxalphanames=99,
backref=false]{biblatex}

\usepackage{tikz-cd}
\usepackage{amssymb}
\usepackage{mathtools}

\newcommand\br[1]{\langle{}#1\rangle{}}
\newcommand\B{\mathrm{B}}
\newcommand\cC{\mathcal{C}}
\newcommand\CH{\mathrm{CH}}
\newcommand\ch{\mathrm{c}}
\newcommand\cS{\mathcal{S}}
\newcommand\ct{\mathrm{ct}}
\newcommand\defeq{\mathrel{\vcenter{\baselineskip0.5ex \lineskiplimit0pt \hbox{.}\hbox{.}}} =}
\newcommand\dto{\dashrightarrow}
\newcommand\F{\mathbb{F}}
\newcommand\GL{\mathrm{GL}}
\newcommand\Gm{\mathbb{G}_{\mathrm{m}}}
\newcommand\hto{\hookrightarrow}
\newcommand\id{\mathrm{id}}
\newcommand\im{\operatorname{im}}
\newcommand\I{^{-1}}
\newcommand\lin{N}
\newcommand\Mb{\overline{\M}}
\newcommand\mul{\mathrm{mul}}
\newcommand\M{\mathcal{M}}
\newcommand\ol[1]{\overline{#1}}
\newcommand\oo{\mathcal{O}}
\newcommand\pbig[1]{\big({#1}\big)}
\newcommand\pBig[1]{\Big({#1}\Big)}
\newcommand\pb{\ar[rd, phantom, "\lrcorner", pos=0]}
\renewcommand\phi{\varphi}
\newcommand\Pj{\mathbb{P}}
\newcommand\pt{\mathrm{pt}}
\newcommand\Q{\mathbb{Q}}
\newcommand\set[2][]{#1\{{#2}#1\}}
\newcommand\sm{\setminus}
\newcommand\sqbig[1]{\big[{#1}\big]}
\newcommand\sqBig[1]{\Big[{#1}\Big]}
\newcommand\sq{\mathrm{sq}}
\newcommand\Sym{\mathrm{Sym}}
\newcommand\toi{\xrightarrow{_\sim}}
\newcommand\tox[2][]{\xrightarrow[#1]{#2}}
\newcommand\Vect{\mathbf{Vect}}
\newcommand\wh[1]{\widehat{#1}}
\newcommand\wt[1]{\widetilde{#1}}
\newcommand\xot[2][]{\xleftarrow[#1]{#2}}
\newcommand\Z{\mathbb{Z}}

\tikzcdset{dgr/.style={cells={nodes={draw, circle,font=\tiny, inner sep=0pt, minimum size=18pt}}, column sep=5pt, row sep=0pt}}

\theoremstyle{definition}
\newtheorem{defn}{Definition}[section]

\theoremstyle{remark}

\theoremstyle{plain}
\newtheorem{thm}[defn]{Theorem}
\newtheorem{lem}[defn]{Lemma}
\newtheorem{propn}[defn]{Proposition}

\let\oldproofname=\proofname
\renewcommand{\proofname}{\rm\bf{\oldproofname}}

\title{\boldmath The integral Chow ring of $\M_2^{\ct}$}
\author{Joseph Helfer, Eric Jovinelly, Eric Larson, Anda Tenie, Chengxi Wang}
\date{December 2024}

\begin{document}
\begin{abstract}
    This paper computes the integral Chow ring of the moduli space $\M_2^\ct$ of stable genus~\(2\) curves of compact type.  This is done by excising boundary strata from $\Mb_2$ one-by-one.  During this process, we determine the Chow rings of all other open strata in $\Mb_2$ with $\Z[\frac{1}{2}]$-coefficients.
\end{abstract}

\maketitle

\section{Introduction}
Chow rings of moduli spaces of stable curves were first investigated by Mumford \cite{mumford1983towards}, who determined \(\CH^*(\Mb_2)\) with rational coefficients and defined certain natural classes in \(\CH^*(\Mb_{g, n})\) known as tautological classes.
At the time, these Chow rings could only be defined with rational coefficients; however, Totaro, Edidin, Graham, and Kresch later developed a theory of integral Chow rings that applies to $\M_g$ and $\Mb_g$ \cite{totaro-BG-chow, edidin1998equivariant, kresch}.

Since Mumford's original work, many rational Chow rings of moduli spaces of curves have been computed: Faber computed those of $\M_3$, $\Mb_3$, and $\M_4$ \cite{C.Faber1990g3,C.Faber1990g4}; Izadi computed that of $\M_5$ \cite{izadi-chow-of-m5}, Penev--Vakil computed that of $\M_6$ \cite{penev-vakil-chow-m6}; and
H.\ Larson--Canning computed that of $\M_g$ for $g \le 9$ as well as those of $\Mb_g$ for $g \le 7$ \cite{canning-larson-chow-m789,canning-larson-chow-and-cohomology-of-mg-arxiv}.  Additionally, many integral Chow rings have been computed:
\begin{itemize}
    \item $\CH^*(\Mb_{0,n})$ for all $n$ was computed by Keel \cite{keel1992intersection};
    \item $\CH^*(\M_{1,n})$ for \(n \leq 10\) was computed by Edidin--Graham and Bishop \cite{edidin1998equivariant,bishop-integral-chow-m-3-to-10};
    \item $\CH^*(\Mb_{1,n})$ for \(n \leq 4\) was computed by Edidin--Graham, Di Lorenzo--Vistoli, and Bishop \cite{edidin1998equivariant,bishop-integral-chow-m-1-3,battistella-di-lorenzo-m-1-n};
    \item $\CH^*(\M_{2,n})$ for $n \le 2$ has been computed by Vistoli, Pernice and Landi, and Landi, respectively \cite{vistoli-chow-ring-of-m2,pernice2023almost,landi-chow-hyperelliptic-arxiv}
    \item $\CH^*(\Mb_{2, n})$ for $n = 0$ and $n = 1$ was computed by E.\ Larson and Di~Lorenzo--Pernice-Vistoli, respectively \cite{larson-chow-m2, di2021stable}.
\end{itemize}

A stable curve is of \emph{compact type} if its Jacobian is proper. In this paper, we compute the integral Chow ring of the moduli stack $\M_2^\ct$ of stable genus~\(2\) curves of compact type. This stack is of particular interest because of its close relationship to abelian surfaces: the Torelli map $J \colon \M_2^\ct \to \mathcal{A}_2$ is a bijection on closed points.

\begin{thm}
Over a field of characteristic \(\neq 2, 3\), we have
\[
  \CH^*(\M_2^\ct) \cong
  \Z[\lambda_1, \lambda_2, \delta_1] / ({
    24\lambda_1^2 - 48\lambda_2,
    24 \delta_1 \lambda_2,
    168 \lambda_1 \lambda_2^2,
    10 \lambda_1 - 2 \delta_1,
    2 \lambda_1^2 - 24 \lambda_2 + \delta_1 \lambda_1 - \delta_1^2
  }).
\]
\end{thm}

Along the way, we also compute the Chow rings with $\Z[1/2]$-coefficients of all other open strata of $\Mb_2$ -- see Theorem~\ref{thm:main-thm} for a precise statement.

\subsection{The Boundary Stratification}
Let us recall the boundary stratification of $\Mb_2$.
To each stable curve $C$ we associate its \textit{dual graph}, whose vertices represent the irreducible components of the curve and whose edges represent the nodes.
In Figure~\ref{fig:m2-boundary-strat}, the names of the strata are displayed to the left, and the dual graphs of the curves they parametrize on the right.
Our convention is that these names always refer to the \emph{closed} strata, i.e., each node in the diagram includes all the ones below it.
\begin{figure}
  \newsavebox\mt\setbox\mt=\hbox{\begin{tikzcd}[dgr]2\end{tikzcd}}
  \newsavebox\dz\setbox\dz=\hbox{\begin{tikzcd}[dgr]1\ar[loop right, -]\end{tikzcd}}
  \newsavebox\di\setbox\di=\hbox{\begin{tikzcd}[dgr]1\ar[r, -]&1\end{tikzcd}}
  \newsavebox\dzz\setbox\dzz=\hbox{\begin{tikzcd}[dgr]0\ar[loop right, -]\ar[loop left,-]\end{tikzcd}}
  \newsavebox\dzzz\setbox\dzzz=\hbox{\begin{tikzcd}[dgr]0\ar[r, -, shift right, shorten=-0.5pt]\ar[r, -, shift left, shorten=-0.5pt]\ar[r, -]&0\end{tikzcd}}
  \newsavebox\dzi\setbox\dzi=\hbox{\begin{tikzcd}[dgr]0\ar[loop left, -]\ar[r, -]&1\end{tikzcd}}
  \newsavebox\dzzi\setbox\dzzi=\hbox{\begin{tikzcd}[dgr]0\ar[loop left, -]\ar[r, -]&0\ar[loop right, -]\end{tikzcd}}
  \[
    \begin{tikzcd}[column sep=10pt, row sep=29pt]
      &\Mb_2\ar[dl, -]\ar[dr, -]\\
      \Delta_0\ar[d, -]\ar[drr, -]&&\Delta_1\ar[d, -]\\
      \Delta_{00}\ar[d, -]\ar[drr, -]&&\Delta_{01}\ar[d, -]\\
      \Delta_{000}&&\Delta_{001}
    \end{tikzcd}
    \quad\quad
    \begin{tikzcd}[column sep=0pt]
      &\usebox{\mt}\ar[dl, -]\ar[dr, -]\\
      \usebox{\dz}\ar[d, -]\ar[drr, -]&&\usebox{\di}\ar[d, -]\\
      \usebox{\dzz}\ar[d, -]\ar[drr, -]&&\usebox{\dzi}\ar[d, -]\\
      \usebox{\dzzz}&&\usebox{\dzzi}
    \end{tikzcd}
  \]
  \caption{The boundary stratification of $\Mb_2$.}
  \label{fig:m2-boundary-strat}
\end{figure}
Notice from the diagram that we have $\M_2^\ct = \Mb_2 \sm \Delta_0$ since the compact type curves are precisely those whose dual graph is a tree.

\subsection{Methods}\label{sec:methods}
Our approach is to use the computation of the integral Chow ring of $\Mb_2$ in \cite{larson-chow-m2} and excise the strata $\Delta_{001}$, $[\Delta_{01}\sm\Delta_{001}]$, $\Delta_{000}$, $[\Delta_{00} \sm (\Delta_{000} \cup \Delta_{001})]$, and $[\Delta_0 \sm (\Delta_{00} \cup \Delta_{01})]$ in order.
The first two of these turn out to be fairly straightforward to excise:
Using the presentation of $\Delta_1$ as a quotient of an affine space by a group action from \cite[\S3.1]{larson-chow-m2}, we find explicit descriptions of the strata $\Delta_{01}$ and $\Delta_{001}$.

Unfortunately, the remaining strata are challenging to excise, because the presentation of $\Mb_2$ in \cite{larson-chow-m2} involves patching together two quotient stacks, one for \(\Delta_1\) and one for its complement.
Thus, while it is easy to understand the pushforward maps
\begin{align*}
\CH^*(\Delta_{000}) &\to \CH^*(\Mb_2 \sm \Delta_1) \\
\CH^*(\Delta_{00} \sm (\Delta_{000} \cup \Delta_{001})) &\to \CH^*(\Mb_2 \sm (\Delta_{000} \cup \Delta_1)) \\
\CH^*(\Delta_0 \sm (\Delta_{00} \cup \Delta_{01})) &\to \CH^*(\Mb_2 \sm (\Delta_{00} \cup \Delta_1)),
\end{align*}
this does not quite determine the needed pushforward maps
\begin{align*}
\CH^*(\Delta_{000}) &\to \CH^*(\Mb_2 \sm \Delta_{01}) \\
\CH^*(\Delta_{00} \sm (\Delta_{000} \cup \Delta_{001})) &\to \CH^*(\Mb_2 \sm (\Delta_{000} \cup \Delta_{01})) \\
\CH^*(\Delta_0 \sm (\Delta_{00} \cup \Delta_{01})) &\to \CH^*(\Mb_2 \sm (\Delta_{00} \cup \Delta_{01})).
\end{align*}

To resolve the ambiguity in the needed pushforward maps, we use a combination of two techniques: First, we use the test family of bielliptic curves described in \cite[\S11]{larson-chow-m2} to determine relations via pullback. Second, we use the Grothendieck--Riemann--Roch formula to establish certain relations in \(\CH^*(\M_2^\ct)\).

\noindent
\textbf{Acknowledgements}
This project began at the 2023 AGNES Summer School, which was held at Brown University and funded by NSF grant DMS-2312088.

The authors would like to thank Dan Abramovich, Juliette Bruce, Samir Canning, Renzo Cavalieri, Melody Chan, Dawei Chen, Emily Clader, Aitor Iribar Lopez, Hannah Larson, Andrea di Lorenzo, Michele Pernice, Brandon Story, Angelo Vistoli, Isabel Vogt, Rachel Webb, Angelina Zheng, and other members of the Brown University math department, for helpful conversations about Chow rings of moduli spaces and/or for organizing the summer school.
Finally, the authors would like to acknowledge the generous support of the National Science Foundation: Eric Jovinelly was supported by an NSF postdoctoral research fellowship DMS-2303335; Eric Larson was supported by NSF grant DMS-2200641.

\tableofcontents

\section{Preliminaries}

\subsection{Assumptions and conventions}\label{subsec:conventions}
Throughout, all schemes and stacks will be over a fixed field $k$ of characteristic different from 2 or 3.  We do not assume $k$ is algebraically closed.

A \textit{global quotient stack} is the quotient $X / G$ of an algebraic space $X$ by a linear algebraic group $G$.
All quotients of spaces by group actions in this paper are stack quotients unless noted otherwise.

In this paper, a \textit{stack} is an Artin stack that admits a stratification by global quotient stacks (see \cite[Def.~3.5.3]{kresch}).
For such stacks, we let $\CH^*$ denote the Chow functor defined in \cite[Definition~2.1.11]{kresch}.
By \cite[Theorem~2.1.12]{kresch}, these Chow groups have familiar properties, such as excision (i.e.\ localization) sequences and a ring structure for smooth stacks.

Given a vector bundle $\mathcal{E} \to X$, we let $\mathbb{P}\mathcal{E}$ denote the space of lines in $\mathcal{E}$.  By \cite[Theorem~2.1.12]{kresch}, the pullback map induces isomorphisms on Chow groups:
\begin{equation}\label{eq:proj-bdl}
    \CH^*(\mathcal{E}) \cong \CH^*(X) \text{ and } \CH^*(\mathbb{P}\mathcal{E}) \cong \CH^*(X)[\zeta]/ (\zeta^r + \ch_1(\mathcal{E}) \zeta^{r-1} + \ldots + \ch_r(\mathcal{E}))
\end{equation}
where $r$ is the rank of $\mathcal{E}$ and $\ch_i(\mathcal{E})\in\CH^*(X)$ are the Chern classes of $\mathcal{E}$.
For a line bundle $\mathcal{L}$ on $X$, the pullback of $\zeta_{\mathcal{E}\otimes\mathcal{L}}$ under the twisting isomorphism $\Pj(\mathcal{E}) \to \Pj(\mathcal{E}\otimes\mathcal{L})$ is $\zeta_{\mathcal{E}}-\ch_1(\mathcal{L})$.

\subsection{Equivariant Computations}
For a smooth global quotient stack $X/G$, the ring $\CH^*(X/G)$ is naturally isomorphic to the $G$-equivariant Chow ring $\CH_G^*(X)$ of \cite{edidin1998equivariant}, and we will use these interchangeably.
Below we recall some facts about equivariant bundles.

Let $G$ be a quasi-compact group scheme acting on a scheme $X$, and consider the map to the quotient stack $\pi \colon X \to X/G$.
The pullback $\pi^*E \to X$ of any vector bundle $E \to X/G$ inherits the structure of a $G$-equivariant bundle -- i.e., a $G$-action on $\pi^*E$ by vector bundle automorphisms making the projection $\pi^*E \to X$ equivariant.
By fpqc descent, this gives an equivalence of categories
\[
  \pi^* \colon \Vect(X/G) \to \Vect_G(X).
\]

That is, we can study vector bundles on $X/G$ in terms of equivariant vector bundles on $X$.
The bundles on $X/G$ which are pulled back along $X/G \to \pt/G$ correspond to the equivariant vector bundles on $X$ of the form $X\times V\tox{\pi_1}X$ for some $G$-representation $V$, where $G$ acts simultaneously on both factors of $X \times V$.
As a special case, vector bundles on $\B G = \pt/G$ correspond to equivariant vector bundles on $\pt$, i.e., $G$-representations.

\subsection{Common groups and classes}
The following notation is consistent with \cite{larson-chow-m2}.

\begin{defn}
    Throughout this paper, let $G$ denote the wreath product
    $$G \defeq (\Gm\times\Gm)\rtimes\Z/2\Z,$$
    where the action of $\Z / 2\Z$ on $\Gm \times \Gm$ permutes its factors.
\end{defn}

One can describe representations of $G$ through the sign representation $\Gamma$ of $\Z/2\Z$ and representations of $\Gm$.
\begin{defn}
  We let $L_n$ denote the 1-dimensional representation of $\Gm$ with weight $n$.
  Set
  \[
    L_{a_1,\ldots, a_n} \defeq L_{a_1} \oplus \ldots \oplus L_{a_n}.
  \]

  Let \( W_{a_1,\ldots, a_n} \) be the representation of \( G \) whose underlying vector space is
  \[
    W_{a_1,\ldots, a_n} \defeq L_{a_1,\ldots, a_n} \oplus L_{a_1,\ldots, a_n}
  \]
  and where \( \Gm \times \Gm \subset G \) acts naturally on \( L_{a_1,\ldots, a_n} \oplus L_{a_1,\ldots, a_n} \) and \( \Z/2\Z \) permutes its factors.

Lastly, we let \( V \) denote the standard representation of \( \GL_2 \) and set \( V_n \defeq \Sym^n V^* \) as well as \( V_n(m) \defeq V_n \otimes (\det V)^{\otimes m} \).
\end{defn}

\begin{defn}\label{defn:classes}
    With the above notation, set
    \begin{align*}
        \alpha_i &= \ch_i(V) \in \CH^*(\B \GL_2),\\
        \beta_i &= \ch_i(W_1) \in \CH^*(\B G),\\
        \gamma &= \ch_1(\Gamma) \in \CH^*(\B \Z/2\Z),
    \end{align*}
    and for $i = 1,2$ let $t_i \in \CH^*\pbig{\B (\Gm \times \Gm)}$ be the pullback of $\ch_1(L_1)$ along the $i$-th projection map $\B(\Gm \times \Gm) \to \B\Gm$.
\end{defn}

Whenever we are working in the Chow ring of a space equipped with a map to \( \B GL_2 \), \( \B G \), or \( \B(\Gm \times \Gm) \), we will denote by \( \alpha_i\) \( \beta_i \), etc., the pullbacks of these classes from \( \CH^*(\B \GL_2) \), etc.
We have the following explicit presentations of these rings:
\begin{lem}[{%
    \cites[Theorem~2.13~and~Lemma~2.12]{totaro-group-cohomology-and-algebraic-cycles}%
    [Theorem~5.2~and~Lemma~7.1]{larson-chow-m2}}%
  ]\label{lem:norm-map-and-presentations}
  We have the presentations
  \[
    \CH^*(\B \GL_2) \cong \Z[\alpha_1,\alpha_2],
    \quad
    \CH^*(\B G) \cong \Z[\beta_1,\beta_2,\gamma] / (2\gamma, \gamma^2 + \beta_1 \gamma),
    \quad
    \CH^*\pbig{\B (\Gm \times \Gm)} \cong \Z[t_1,t_2].
  \]
  Moreover, the \emph{norm} (or \emph{transfer} map) from \( \Gm \times \Gm \) to \( G \) -- in other words, the pushforward \( \pi_* \colon \CH^*(\Gm \times \Gm) \to \CH^*(G) \) along the index 2 inclusion \( \Gm \times \Gm \hto G \) -- is given by
  \begin{align*}
    \pi_*(1) = & 2 &
    \pi_*(t_1^a) = & \beta_1 \pi_*(t_1^{a-1}) - \beta_2 \pi_*(t_1^{a-2}) \text{ for } a\geq 2\\
    \pi_*(t_1) = & \beta_1 + \gamma &
    \pi_*(t_1^a t_2^b) = & \beta_2^{\min(a,b)} \pi_*(t_1^{|a-b|}),
  \end{align*}
  while the pullback along \( \pi \) is given by
  \[
    \pi^*\beta_1 = t_1 + t_2,
    \qquad
    \pi^*\beta_2 = t_1 t_2,
    \qquad
    \pi^*\gamma = 0.
  \]
\end{lem}
Note that the image of \( \pi_* \colon \CH^*(\B G) \to \CH^*\pbig{\B (\Gm \times \Gm}) \cong \Z[t_1,t_2] \) is the subring generated by \( t_1 + t_2 \) and \( t_1 t_2 \), and \( \Z[t_1,t_2] \) is generated as a module over this ring by \( 1 \) and \( t_1 \) (or by \( 1 \) and \( t_2 \)).
In particular (by \S\ref{subsec:push-pull} below), pushing forward an ideal under \( \pi_* \) amounts to pushing forward a set of generators of this ideal and their products with \( t_1 \) (or with \( t_2 \)).

\subsection{Push pull}\label{subsec:push-pull}
We recall the ``push-pull'' property for Chow rings: this says that, given a proper morphism $f \colon X\to Y$ of stacks with \( X \) and \( Y \) smooth, $f_*\colon\CH^*(X)\to\CH^*(Y)$ is a $\CH^*(Y)$-module homomorphism.
Hence, more generally, given a second map $g\colon Y\to Z$ (with $Z$ smooth), $f_*$ is a $\CH^*(Z)$-module homomorphism.

We will be concerned with finding the images of several such pushforward morphisms $f_*$.
To this end, we note that if $S \subset \CH^*(X)$ generate $\CH^*(X)$ as a $\CH^*(Z)$-module, then $f_*(S)$ generates the image of $f_*$ as an $\CH^*(Z)$-module.
Thus, in particular, if $(gf)^* \colon \CH^*(Z) \to \CH^*(X)$ is surjective, then the image of $f_*$ is generated as an ideal by the image $f_*(1)$ of the fundamental class of $X$.

This situation will arise for us with $Z = \B H$ for some group $H$, and $X$ and $Y$ open subsets of vector bundles over $Z$.

A related fact we will be using repeatedly is that ``flat pullback commutes with projective pushforward'' \cite[\S2.2]{kresch}: given a pullback square of stacks as below on the left with the horizontal morphisms projective and the vertical morphisms flat, the resulting square of Chow rings on the right is commutative.
\begin{equation}\label{eq:pb-square-ch-commutes}
  \begin{tikzcd}
    X' \ar[r, "f'"] \ar[d, "h'"'] \pb & X \ar[d, "h"] \\
    Y' \ar[r, "f"] & Y
  \end{tikzcd}
  \qquad
  \qquad
  \begin{tikzcd}
    \CH^*(X') \ar[r, "f'_*"]
    \ar[from=d, "(h')^*"] \pb &
    \CH^*(X) \ar[from=d, "h^*"'] \\
    \CH^*(Y') \ar[r, "f_*"] & \CH^*(Y)
  \end{tikzcd}
\end{equation}

\subsection{Chow rings of \texorpdfstring{\boldmath $\Mb_2$}{M₂} and \texorpdfstring{\boldmath $\Delta_1$}{Δ₁}}\label{subsec:d1-pres}
We recall the main result of \cite{larson-chow-m2}, which is the starting point of our computations.

\begin{thm}[Theorem 1.1, \cite{larson-chow-m2}]\label{thm:chow-M2bar}
Over any base field of characteristic distinct from 2 and 3, the Chow ring of the moduli space of stable curves of genus 2 is given by
\[
  \CH^*(\Mb_2)\cong
  \Z[\lambda_1, \lambda_2, \delta_1 ] / (
  24\lambda_1^2 - 48\lambda_2,
  20\lambda_1\lambda_2 - 4\delta_1\lambda_2,
  \delta_1^3 + \delta_1^2 \lambda_1,
  2\delta_1^2 + 2\delta_1 \lambda_1
  ),
\]
where $\lambda_1,\lambda_2$ denote the Chern classes of the Hodge line bundle and $\delta_1$ denotes the class of the boundary substack parameterizing curves with a disconnecting node.

In this basis, the class of the boundary substack parameterizing curves with a nonseparating node is given by $\delta_0 = 10\lambda_1 - 2\delta_1$.
\end{thm}

\begin{defn}
  As in Definition~\ref{defn:classes}, we will write \( \delta_0 \), \( \delta_1 \), \( \lambda_1 \), and \( \lambda_2 \) for the pullbacks of the respective class from \( \CH^*(\Mb_2) \), whenever this makes sense.
\end{defn}

\begin{thm}[{\cite[\S9]{larson-chow-m2}}]\label{thm:ch-of-d1-delta}
  The Chow ring of \( \Delta_1 \) is given by
  \[
    \CH^*(\Delta_1)\cong
    \Z[\lambda_1,\lambda_2,\gamma]/
    (
    2\gamma,\gamma^2+\lambda_1\gamma,24\lambda_1^2-48\lambda_2,
    24\lambda_1\lambda_2
    ).
  \]
\end{thm}

In \cite[\S9]{larson-chow-m2}, It will be more convenient for us to present this in terms of the generators \( \lambda_1 \), \( \lambda_2 \), and \( \delta_1 \).
By \cite[Lemma~6.1]{larson-chow-m2} we have that \( \delta_1 = \gamma - \lambda_1 \) on \( \Delta_1 \), and hence:
\begin{equation}\label{eq:ch-of-delta1}
  \begin{split}
    \CH^*(\Delta_1)&\cong
    \Z[\lambda_1,\lambda_2,\delta_1]/ (
    2(\delta_1+\lambda_1),
    \delta_1(\delta_1+\lambda_1),
    24\lambda_1^2-48\lambda_2,
    24\lambda_1\lambda_2
    )
  \end{split}
\end{equation}

\section{Excising \texorpdfstring{$\Delta_{001}$}{Δ₀₀₁} and \texorpdfstring{$\Delta_{01}$}{Δ₀₁}}
In this section, we carry out the first two steps of the program outlined in \S\ref{sec:methods}, by excising the loci $\Delta_{001}$ and $\Delta_{01} \sm \Delta_{001}$.
Along the way, we will also determine the Chow rings of $\Delta_1 \sm \Delta_{001}$ and $\Delta_1 \sm \Delta_{01}$:
\begin{propn}\label{propn:delta-summary}
  The rings $\CH^*(\Delta_1 \sm \Delta_{001})$, $\CH^*(\Delta_{1} \sm \Delta_{01})$, $\CH^*(\Mb_2 \sm \Delta_{001})$, and $\CH^*(\Mb_2 \sm \Delta_{01})$ are given by the following presentations:
\[
    \arraycolsep=2pt
    \renewcommand\arraystretch{1.5}
    \begin{array}{ll}
    \CH^*(\Delta_1 \sm \Delta_{001})
    & \cong
    \CH^*(\Delta_1) / ({
        144 \lambda_2
    })
    \\ & \cong
    \Z[\lambda_1,\lambda_2,\delta_1] / ({
        2(\delta_1+\lambda_1),
        \delta_1(\delta_1+\lambda_1),
        24\lambda_1^2-48\lambda_2,
        24\lambda_1\lambda_2,
        144\lambda_2
    })
    \\[10pt]
    \CH^*(\Delta_{1} \sm \Delta_{01})
    & \cong
    \CH^*(\Delta_1 \sm \Delta_{001}) / ({
        12 \delta_1,
        24 \lambda_2
    })
    \\ & \cong
    \Z[\lambda_1,\lambda_2,\delta_1] / ({
        2(\delta_1+\lambda_1),
        \delta_1(\delta_1+\lambda_1),
        12\lambda_1,
        24\lambda_2
    })
    \\[10pt]
    \CH^*(\Mb_2 \sm \Delta_{001})
    & \cong
    \CH^*(\Mb_2) / ({
        144 \delta_1 \lambda_2
    })
    \\ & \cong
    \Z[\lambda_1, \lambda_2, \delta_1 ] / ({
        24\lambda_1^2 - 48\lambda_2,
        20\lambda_1\lambda_2 - 4\delta_1\lambda_2,
        \delta_1^3 + \delta_1^2 \lambda_1,
        2\delta_1^2 + 2\delta_1 \lambda_1,
        144\delta_1\lambda_2
    })
    \\[10pt]
    \CH^*(\Mb_2 \sm \Delta_{01})
    & \cong
    \CH^*(\Mb_2 \sm \Delta_{001}) / ({
        12 \delta_1 (2\lambda_1 + \delta_1),
        24 \delta_1 \lambda_2
    })
    \\ & \cong
    \Z[\lambda_1, \lambda_2, \delta_1 ] / ({
      24\lambda_1^2 - 48\lambda_2,
      20\lambda_1\lambda_2 - 4\delta_1\lambda_2,
      \delta_1^3 + \delta_1^2 \lambda_1,
    } \\ & \hspace{7.5em} { 
      2\delta_1^2 + 2\delta_1 \lambda_1,
      12\delta_1\lambda_1,
      24\delta_1\lambda_2
    }).
    \end{array}
\]
\end{propn}

The rest of this section will be devoted to the proof of this proposition.
The computation proceeds in four steps, by using the following excision sequences.
In other words, our starting point is the Chow rings of $\Delta_1$ and $\Mb_2$ by \cite{larson-chow-m2} (see \S\ref{subsec:d1-pres}).
In each case, having already computed the second term in the sequence, we compute the third term by finding the image of the first arrow.
\begin{align*}
  \text{(i)} &
  & \CH^*(\Delta_{001}) && \longrightarrow &
  & \CH^*(\Delta_1) && \longrightarrow &
  & \CH^*(\Delta_1 \sm \Delta_{001}) && \longrightarrow &
  & 0
  \\
  \text{(ii)} &
  & \CH^*(\Delta_{01} \sm \Delta_{001}) && \longrightarrow &
  & \CH^*(\Delta_{1} \sm \Delta_{001}) && \longrightarrow &
  & \CH^*(\Delta_{1} \sm \Delta_{01}) && \longrightarrow &
  & 0
  \\
  \text{(iii)} &
  & \CH^*(\Delta_{001}) && \longrightarrow &
  & \CH^*(\Mb_2) && \longrightarrow &
  & \CH^*(\Mb_2 \sm \Delta_{001}) && \longrightarrow &
  & 0
  \\
  \text{(iv)} &
  & \CH^*(\Delta_{01} \sm \Delta_{001}) && \longrightarrow &
  & \CH^*(\Mb_2 \sm \Delta_{001}) && \longrightarrow &
  & \CH^*(\Mb_2 \sm \Delta_{01}) && \longrightarrow &
  & 0
\end{align*}

For (i) and (ii), we take advantage of the fact that we are working entirely inside of $\Delta_1$, which has a nice presentation as a global quotient stack \cite[\S3.1]{larson-chow-m2}.
For (iii) and (iv), we reuse the computations in (i) and (ii), factoring the leftmost maps as $\CH^*(\Delta_{001}) \to \CH^*(\Delta_1) \to \CH^*(\Mb_2)$ and $\CH^*(\Delta_{01} \sm \Delta_{001}) \to \CH^*(\Delta_1 \sm \Delta_{001}) \to \CH^*(\Mb_2 \sm \Delta_{001})$, respectively.

\subsection{Quotient presentations of \texorpdfstring{\boldmath $\Delta_{01}$}{Δ₀₁} and \texorpdfstring{$\Delta_{001}$}{Δ₀₀₁}}\label{subsec:quotient-presentations}
We may represent $\overline{\mathcal{M}}_{1,1}$ as a global quotient stack $\Pj(4,6) \cong ( L_{4,6} \sm 0) / \Gm$ using the Weierstrass form for elliptic curves \cite[Section~5.4]{edidin1998equivariant}.
As in \cite[\S3.1]{larson-chow-m2}, we thus obtain a presentation of $\Delta_1 \cong \Sym^2 (\Mb_{1,1})$ as a global quotient stack
\begin{equation}\label{eq:delta1-pres}
  \Delta_{1} \cong
  \pbig{W_{4,6} \sm (L_{4,6} \times 0 \cup 0 \times L_{4,6})}/G \cong
  \pbig{(L_{4,6} \sm 0)^2}/G.
\end{equation}
The ring \( \CH^*(\Delta_1) \) is thus a quotient of \( \CH^*(\B G) \cong \Z[\beta_1, \beta_2, \gamma]/(2\gamma,\gamma^2 + \beta_1\gamma) \).
With respect to this presentation, the Hodge bundle on \( \Delta_1 \) is the pullback under \( \Delta_1 \to BG \) of \( W_1 \) \cite[\S3.1]{larson-chow-m2}, and hence the classes \( \lambda_i \) on \( \Delta_1 \subset \Mb_2 \) coincide with \( \beta_i \).

Inside of \( L_{4,6} \) (the parameter space for Weierstrass curves \( y^2 = x^3+ax+b \)), we have the discriminant locus \( D = \set{4a^3+27b^2=0} \).
With respect to the presentation \eqref{eq:delta1-pres} of \( \Delta_1 \), the strata \( \Delta_{01},\Delta_{001} \subset \Delta_1 \) are given by
\[
  \Delta_{01} =
  \sqBig{\pBig{(D \sm 0)\times(L_{4,6} \sm 0)}\cup
         \pBig{L_{4,6} \sm 0)\times(D \sm 0)}}/G
    \qquad
    \text{and}
    \qquad
  \Delta_{001}=
  (D \sm 0)^2/G.
\]

\subsection{Chow ring of \texorpdfstring{\boldmath $\Delta_1 \sm \Delta_{001}$}{Δ₀₁ \ Δ₀₀₁}} \label{subsec:chow-of-d1-d001}
We use the excision sequence (i) displayed above, or equivalently, by \S\ref{subsec:quotient-presentations}, the sequence:
\begin{equation}\label{eq:001-sequence-G}
  \CH^*_G\pbig{(D \sm 0)^2} \to
  \CH^*_G\pbig{(L_{4,6} \sm 0)^2} \to
  \CH^*_G\pbig{(L_{4,6} \sm 0)^2 \sm (D \sm 0)^2} \to
  0.
\end{equation}
Since we have a $\Gm$-equivariant isomorphism $L_{2} \sm 0 \cong D \sm 0$, given by $t \mapsto (3 t^2,2 t^3)$, it follows from \S\ref{subsec:push-pull} that the image of the first map in \eqref{eq:001-sequence-G} is generated as an ideal by the image of the fundamental class.
Referring to the statement of Proposition~\ref{propn:delta-summary}, we wish to show that this image is $144 \lambda_2$.

The curve $D \sm 0\subset L_{4,6} \sm 0$ is the vanishing locus of the function $4a^3+27b^2$, which can be considered as an equivariant section of the equivariant vector bundle $(L_{4,6} \sm 0)\times_{\B \Gm} L_{12}$, since $\lambda\in\Gm$ acts on $L_{4,6}$ by $a\mapsto\lambda^4a$ and $b\mapsto\lambda^6b$, and hence $4a^3+27b^2 \mapsto \lambda^{12}(4a^3+27b^2)$.
Therefore, $[D \sm 0]_G = \ch_1^G(L_{12}) \in \CH_G^*(L_{4,6} \sm 0)$.

Similarly, $(D \sm 0)^2\subset(L_{4,6} \sm 0)^2$ is the vanishing locus of an equivariant section of $(L_{4,6} \sm 0) \times_{\B G} W_{12}$.
Using \cite[Lemma~7.2]{larson-chow-m2}, we conclude
\[
  [(D \sm 0)^2]_G =
  \ch_2(W_{12}) =
  144 \beta_2 =
  144 \lambda_2 \in
  \CH_G^*\pbig{(L_{4,6} \sm 0)^2},
\]
as required.

\subsection{Chow ring of \texorpdfstring{\boldmath $\Delta_1 \sm \Delta_{01}$}{Δ₁ \ Δ₀₁}}\label{subsec:chow-of-d1-d01}
This time, we will use the exact sequence (ii) above, or equivalently, by \S\ref{subsec:quotient-presentations}, the sequence
\[
\begin{split}
  \CH_G^*\pBig{
    \pbig{(D \sm 0)\times(L_{4,6} \sm D)}\cup
    \pbig{(L_{4,6} \sm D)\times(D \sm 0)}
  }
  \to \qquad \\ \qquad \to
  \CH_G^*\pbig{(L_{4,6} \sm 0)^2 \sm (D \sm 0)^2} \to
  \CH_G^*\pbig{(L_{4,6} \sm D)^2}\to0.
\end{split}
\]
Observe that we have a commutative square of stacks
\[
  \begin{tikzcd}
    \sqBig{(D \sm 0)\times(L_{4,6} \sm D)}/(\Gm\times\Gm)\ar[d, "\sim" sloped]\ar[r]&
    {\pbig{(L_{4,6} \sm 0)^2 \sm (D \sm 0)^2}/(\Gm\times\Gm)}\ar[d]\\
    \sqBig{
      \pbig{(D \sm 0)\times(L_{4,6} \sm D)}\cup
      \pbig{(L_{4,6} \sm D)\times(D \sm 0)}
    }/G
    \ar[r]&
    {\pbig{(L_{4,6} \sm 0)^2 \sm (D \sm 0)^2}/G}
  \end{tikzcd}
\]
where the leftmost vertical arrow is an isomorphism.
Hence, to compute the image of pushforward along the bottom horizontal morphism, we may instead compute the image of the composite of pushforwards along the top and right morphisms.

We begin with
\begin{equation}\label{eq:top-of-square-map}
  \CH^*_{\Gm\times\Gm}\pbig{(D \sm 0)\times(L_{4,6} \sm D)} \to
  \CH^*_{\Gm\times\Gm}\pbig{(L_{4,6} \sm 0)^2 \sm (D \sm 0)^2},
\end{equation}
the image of which is generated as an ideal by the pushforward of the fundamental class by \S\ref{subsec:push-pull}.

This image is the fundamental class of $(D \sm 0)\times(L_{4,6} \sm D) \subset \pbig{(L_{4,6} \sm 0)^2 \sm (D \sm 0)^2}$, which is the vanishing locus of an equivariant section of $L_{12}$ pulled back from the first factor.
Hence the class in question is just $12t_1$ (with $t_1$ as in Definition~\ref{defn:classes}), and thus the image of \eqref{eq:top-of-square-map} is $(12t_1)$.

It remains to find the image of $(12t_1)$ under the pushforward map
\begin{equation}
  \CH^*_{\Gm\times\Gm}\pbig{(L_{4,6} \sm 0)^2 \sm (D \sm 0)^2} \to
  \CH^*_{G}\pbig{(L_{4,6} \sm 0)^2 \sm (D \sm 0)^2},
\end{equation}
i.e., under the norm map \( \pi_* \colon \CH^*\pbig{\B(\Gm \times \Gm)} \to \CH^*(\B G) \).
By Lemma~\ref{lem:norm-map-and-presentations}, this image is generated by
\[
  \pi_*(12t_1)=
  12\pi_*(t_1)=
  12(\beta_1+\gamma)
  \quad
  \text{ and }
  \quad
  \pi_*(12t_1t_2)=
  12\pi_*(t_1t_2)=
  12\beta_2\pi_*(t_1^0)=
  24\beta_2.
\]
In conclusion, the image of
\[
  \CH^*(\Delta_{01} \sm \Delta_{001}) \to \CH^*(\Delta_1 \sm \Delta_{001})
\]
is the ideal \( (12(\lambda_1+\gamma),24\lambda_2) = (12(2\lambda_1+\delta_1),24\lambda_2) = (12 \lambda_1, 24 \lambda_2) \), this last equality coming from the fact that \( 2\delta_1 = - 2\lambda_1 \) in \( \CH^*(\Delta_1 \sm \Delta_{001}) \).
This gives the presentation of \( \CH^*(\Delta_1 \sm \Delta_{01}) \) claimed in Proposition~\ref{propn:delta-summary}.

\subsection{Chow rings of \texorpdfstring{\boldmath $\Mb_2 \sm \Delta_{001}$}{M₂ \ Δ₀₀₁} and \texorpdfstring{$\Mb_2 \sm \Delta_{01}$}{M₂ \ Δ₀₁}} \label{subsec:chow-of-m2-minus-delta01}
We see from Theorems~\ref{thm:chow-M2bar}~and~\ref{thm:ch-of-d1-delta} that the pullback map \( i^* \colon \CH^*(\Mb_2) \to \CH^*(\Delta_1) \) is surjective, and hence that the pushforward map \( \CH^*(\Delta_1) \to \CH^*(\Mb_2) \) is determined by the rule \( (i^*x) \cdot 1 \mapsto x \cdot \delta_1 \).

We factor the first map of the excision sequence (iii) as
\[
  \CH^*(\Delta_{001}) \to \CH^*(\Delta_{1}) \to \CH^*(\Mb_2).
\]
For the first of these, we already computed its image in \S\ref{subsec:chow-of-d1-d001} to be \( (144\lambda_2) \).
Hence, the image of the composite is \( (144\delta_1\lambda_2) \), giving the claimed presentation of \( \CH^*(\Mb_2 \sm \Delta_{001}) \).

Next, we use the excision sequence (iv) and again factor the first map as
\[
  \CH^*(\Delta_{01} \sm \Delta_{001}) \to \CH^*(\Delta_{1} \sm \Delta_{001}) \to \CH^*(\Mb_2 \sm \Delta_{001}).
\]
The image of the first of these maps was determined in \S\ref{subsec:chow-of-d1-d01} to be \( (12(\lambda_1+\gamma),24\lambda_2) \).
Hence, the image of the composite is \( (12(2\lambda_1+\delta_1)\delta_1,24\delta_1\lambda_2) \), giving the claimed presentation of \( \CH^*(\Mb_2 \sm \Delta_{01}) \).

\section{Bielliptic Curves}\label{sec:bielliptic}
In \cite[\S11]{larson-chow-m2}, the stack of bielliptic curves $B$ is constructed as an open substack of the quotient stack
\[
  \ol{B} \defeq
  \sqbig{V_3(1) \times \pbig{V_1(-1) \boxtimes W_{-2}}} / (\GL_2 \times G),
\]
where $\boxtimes$ is the external product of representations, and where $G$ acts trivially on $V_3(1)$.  We view elements of $V_3(1)$ as cubic polynomials $H = ex^3 + fx^2y + gxy^2 + hy^3$ and elements of $V_1(-1) \boxtimes W_{-2}$ as pairs of linear polynomials $(\lin_1,\lin_2) = (ax + by, cx + dy)$.

Specifically, let $\wh{B} \subset V_3(1) \times \pbig{V_1(-1) \boxtimes W_{-2}}$ be the open subscheme consisting of those $(H,\lin_1,\lin_2)$ such that $\lin_1 \cdot \lin_2 \cdot H$ is non-zero and does not have a triple root.
We note that $\wh{B}$ is invariant under the action of $\GL_2 \times G$ and set
\[
  B = \wh{B} / (\GL_2 \times G).
\]
In \S\ref{subsec:univ-family}, we will construct a family of stable curves of genus $2$ over $B$, i.e., a map $\pi \colon B \to \Mb_2$. In fact, this realizes $B$ as the moduli stack of bielliptic curves, as explained in \cite[\S11.1]{larson-chow-m2}, though we will not need this fact.

In \cite[Theorem~11.3]{larson-chow-m2}, the Chow ring of $B$ is computed to be:
\begin{equation}\label{eq:chb-m2-pres}
  \begin{split}
    \CH^*(B)&\cong
    \Z[\gamma,\delta_1,\lambda_1,\lambda_2]/
    (
    2 \gamma,
    \gamma^2 + \lambda_1 \gamma,
    \delta_1^2 + \delta_1 \gamma + 8 \lambda_1^2 - 12 \lambda_2,
    24 \lambda_1^2 - 48 \lambda_2,\\
    &\hspace{8.3em} 
    2 \delta_1^2 + 2 \lambda_1 \delta_1,
    20 \lambda_1 \lambda_2 - 4 \delta_1 \lambda_2,
    8 \lambda_1^3 - 8 \lambda_1 \lambda_2
    ).
  \end{split}
\end{equation}
Here, $\gamma$ is pulled back along $B \to \B G$ and $\delta_1$, $\lambda_1$, and $\lambda_2$ are pulled back along the map $\pi \colon B \to \Mb_2$ which we construct below.
In the remainder of \S\ref{sec:bielliptic}, we prove the following proposition, which we will use to probe $\CH^*(\M_2^\ct)$ in \S\ref{sec:excising-delta0}.

\begin{propn}\label{propn:biell-summary}
  The rings $\CH^*(B \sm \pi\I(\Delta_{000}))$, $\CH^*(B \sm \pi\I(\Delta_{00}))$, and $\CH^*(B \sm \pi\I(\Delta_{0}))$ have the following presentations.
  \[
    \begin{split}
      \CH^*(B \sm \pi^{-1}(\Delta_{000})) &
      \cong \CH^*(B) / ({
        16\lambda_1^2 - 20 \delta_1\lambda_1
      })\\
      &\cong
      \Z[\gamma,\delta_1,\lambda_1,\lambda_2] / ({
        2 \gamma,
        \gamma^2 + \lambda_1 \gamma,
        \delta_1^2 + \delta_1 \gamma + 8 \lambda_1^2 - 12 \lambda_2,
        24 \lambda_1^2 - 48 \lambda_2,
      } \\ & \hspace{8.3em} { 
        2 \delta_1^2 + 2 \lambda_1 \delta_1,
        20 \lambda_1 \lambda_2 - 4 \delta_1 \lambda_2,
        8 \lambda_1^3 - 8 \lambda_1 \lambda_2,
        16\lambda_1^2 - 20 \delta_1\lambda_1
      })\\
      \CH^*(B \sm \pi^{-1}(\Delta_{00}))
      &\cong
      \CH^*(B \sm \pi\I(\Delta_{000})) / ({
        4\lambda_1 - 2\delta_1
      })\\
      &\cong \CH^*(B) / ({
        4\lambda_1 - 2 \delta_1
      })\\
      &\cong
      \Z[\gamma,\delta_1,\lambda_1,\lambda_2] / ({
        2 \gamma,
        \gamma^2 + \lambda_1 \gamma,
        \delta_1^2 + \delta_1 \gamma + 8 \lambda_1^2 - 12 \lambda_2,
      } \\ &\hspace{8.3em} { 
        6 \lambda_1 \delta_1,
        12 \lambda_1 \lambda_2,
        8 \lambda_1^3 - 8 \lambda_1 \lambda_2,
        4\lambda_1-2\delta_1
      }),
      \\
      \CH^*(B \sm \pi^{-1}(\Delta_{0}))
      &\cong
      \CH^*(B \sm \pi\I(\Delta_{00})) / ( {
        6\lambda_1, \delta_1^2 + \delta_1\lambda_1
      })
      \\ &\cong
      \CH^*(B) / ({
        4\lambda_1 -2\delta_1, 6\lambda_1, \delta_1^2 + \delta_1\lambda_1
      })
      \\ &\cong
      \Z[\gamma,\delta_1,\lambda_1,\lambda_2] / ({
        2 \gamma,
        \gamma^2 + \lambda_1 \gamma,
        \delta_1^2 + \delta_1 \gamma + 2 \lambda_1^2 - 12 \lambda_2,
      } \\ & \hspace{8.3em} { 
        2 \lambda_1^3 - 2 \lambda_1 \lambda_2,
        4\lambda_1-2\delta_1,
        6\lambda_1,
        \delta_1^2+\delta_1\lambda_1
      }).
    \end{split}
  \]
\end{propn}



\subsection{The universal family}\label{subsec:univ-family}
As in \cite{larson-chow-m2}, we equip $B$ with a map $\pi \colon B \to \Mb_2$ by constructing a family of stable curves over it as follows.
Consider the following family of curves in weighted projective space $\Pj(1,1,2,2)$ with coordinates $[x:y:z:w]$:
\begin{align}\label{eq:univ-family}
\begin{split}
    z^2 &= \lin_1 \cdot H\\
    w^2 &= \lin_2 \cdot H,
\end{split}
\end{align}
i.e., the subspace of $\wh{B} \times \Pj(1,1,2,2)$ defined by the above equations; we denote this subspace by $\wh{X}$.
Note that $\wh{X}$ is disjoint from the locus $\set{x=y=0} \subset \Pj(1,1,2,2)$ containing all points with non-trivial isotropy.
In particular, the projection $[x:y] \colon \Pj(1,1,2,2) \dto \Pj^1$ restricts to a regular map $\wh{X} \to \Pj^1$, which factors through each of the maps $[x:y:z],[x:y:w]\colon \Pj(1,1,2,2) \to \Pj(1,1,2)$.  This results in a cartesian square of schemes over $\wh{B}$:
\[
  \begin{tikzcd}
    \wh{X} \ar[r, ""] \ar[d, ""'] \pb & E_1 \ar[d, ""] \\
    E_2 \ar[r, ""] & \wh{B} \times \Pj^1.
  \end{tikzcd}
\]
Here, $E_1$ and $E_2$ are the families in $\wh{B} \times \Pj(1,1,2)$ defined by $z^2 = \lin_1 \cdot H$ and by $w^2 = \lin_2 \cdot H$, respectively.
Each $E_i$ is a family of nodal elliptic curves over $\wh{B}$ and the maps $E_i \to \wh{B} \times \Pj^1$ are families of degree two covers.
The fiber of $E_1 \to \wh{B}$ (resp.\ $E_2 \to \wh{B}$) acquires a singularity when $\lin_1 \cdot H$ (resp.\ $\lin_2 \cdot H$) has a double root.

By considering the natural actions of $\GL_2$ on $(x,y)$ and of $G$ on $(z,w)$, we obtain an action of $\GL_2 \times G$ on $\wh{X}$, with respect to which the projection map $\wh{X} \to \wh{B}$ is equivariant.
The quotient
\[
  \wh X / (\GL_2 \times G) \to B
\]
is \emph{almost} our desired universal family.
The problem is that the fibres are not genus 2 stable curves, so we first need to resolve them by performing the following weighted blowup.

Consider the rational map $\wh{X} \dto \Pj(2,2,3)$ given by
\[
  \pbig{
    (a, \ldots, h),
    [x:y:z:w]
  }
  \mapsto
  [z:w:H].
\]
(where $H = ex^3+fx^2y+gx^2y+hy^3$ as above) and define $\wh{Y} \subset \Pj(1,1,2,2) \times V_3(1) \times \pbig{V_1(-1) \boxtimes W_1} \times \Pj(2,2,3)$ to be the closure of the graph of this map.
Note that, exceptionally, we are considering the weighted projective space $\Pj(2,2,3)$ not as a stack but as a scheme.

Explicitly, $\wh{Y}$ is defined by the equations
\begin{equation}
  \begin{split}
    z^2 &= \lin_1 \cdot H \\
    w^2 &= \lin_2 \cdot H \\
    t w & = u z \\
    t^2 \lin_2 & = u^2 \lin_1
  \end{split}
  \qquad
  \begin{split}
    t^3 H^2 & = v^2 z^3 \\
    t^2 u H^2 & = v^2 z^2 w \\
    t u^2 H^2 & = v^2 z w^2 \\
    u^3 H^2 & = v^2 w^3
  \end{split}
  \qquad
  \begin{split}
    t^3 z & = v^2 \lin_1^2 \\
    t^2 u w & = v^2 \lin_1 \lin_2 \\
    t u^2 z & = v^2 \lin_1 \lin_2 \\
    u^3 w & = v^2 \lin_2^2
    , 
  \end{split}
\end{equation}
where $[t:u:v]$ are the coordinates on $\Pj(2,2,3)$.
Again, $\wh{Y}$ inherits a $\GL_2\times G$ action and we have an equivariant map $\wh{Y} \to \wh{B}$.
Taking quotients, the resulting map $Y \defeq \wh{Y} / (\GL_2 \times G) \to B$ is a family of nodal curves, as we will show below.
We summarize the situation:
\begin{lem}\label{lem:loci-in-b}
  The family \( Y \to B \) is a family of genus 2 stable curves, and  the resulting map $\pi \colon B \to \Mb_2$ satisfies the following properties:
  \begin{itemize}
  \item $\pi\I(\Delta_0) \subset B$ is the locus of $(H,\lin_1,\lin_2) \in B$ such that $\lin_1 \mid H$ or $\lin_2 \mid H$ or $H$ has a double root.
  \item $\pi\I(\Delta_{00}) \subset B$ is the locus of $(H,\lin_1,\lin_2) \in B$ such that $\lin_1, \lin_2 \mid H$ or $H$ has a double root.
  \item $\pi\I(\Delta_{000}) \subset B$ is the locus of $(H,\lin_1,\lin_2) \in B$ such that $H$ has a double root, and one of $\lin_1, \lin_2$ divides $H$.
  \item $\pi\I(\Delta_1)$ is the locus of $(H,\lin_1,\lin_2)$ such that $\lin_1$ and $\lin_2$ differ by scaling.
  \end{itemize}
\end{lem}
\begin{proof}
  Fix a geometric point \( p = (H,\lin_1,\lin_2) \) of \( B \), consider the geometric fibre \( C \) over \( B \), and write \( C \to C_i \to \Pj^1 \), \( i = 1,2 \), for the intermediate semistable elliptic curves.  We claim \(C\) is a degree~\(2\) cover of \(C_1\), ramified exactly over the preimages of the root of \(N_2\).  Indeed, recall $C_1$ is defined by the equation $z^2 = N_1 \cdot H$.  There are four different values of $(z,w)$ for any point in $\Pj^1$ not contained in $V(H\cdot N_1 \cdot N_2)$.  Similarly, there are two distinct values of $w$ over $V(N_1)$, but only one value for $z$.  Lastly, as the blow up $\wh{Y} \to Y$ keeps track of the ratio of $z$ and $w$ when $H$ approaches~$0$, there are two distinct values for $[z:w]$.  This implies our claim.  In particular, the Riemann--Hurwitz formula implies that \(C\) is a stable curve of genus \(2\).

  Which of the above strata of \( B \) the point \( p \) lies in depends on how many nodes the curve \( C \) has, and whether or not they are disconnecting.
  The analysis in the previous paragraph determines the number of nodes of \( C \); namely that
  \begin{itemize}
      \item $C$ acquires two nodes in the fibre over a double root of $H$, and
      \item $C$ acquires a single node in the fibre over a common zero of any two of $\lin_1$, $\lin_2$ and $H$,
  \end{itemize}
  which is as required.
  It just remains to see that the resulting node in \( C \) is disconnecting if and only if we are in the case in which the roots of \( \lin_1 \) and \( \lin_2 \) collide.

  To verify this, we may suppose that only one pair of roots collides, since if we fix our attention on one node, we can smooth the other nodes without changing whether or not it is disconnecting.

  Suppose that the roots of \( N_1 \) and \( N_2 \) coincide.
  Then in \eqref{eq:univ-family}, we have \( z^2 = \mu w^2 \) for some constant $\mu$, and hence \( C \) is reducible.
  Since $C$ has a single node, this node must be disconnecting.

  Suppose that the roots of \( N_1 \) and \( N_2 \) are distinct.
  Without loss of generality, we may assume that \( N_2 \) does not divide \( H \).
  Then \( N_1 \cdot H \) has at least one simple root, over which \( C_1 \to \Pj^1 \) is simply branched; thus \( C_1 \) is irreducible by monodromy considerations.
  Similarly, \( C \to C_1 \) is simply branched point over the two preimages in \( C_1 \) of the root of \( N_2 \), and hence is also irreducible.
\end{proof}

\subsection{General strategy for the excision computations}\label{subsec:biell-excision-strategy}
We now wish to excise various closed substacks from $B$ (namely the closed substacks $\pi\I(\Delta_\star)$ for $\star \in \set{0,00,000}$ appearing in the above proposition) and compute the resulting Chow rings.
We recall the general strategy for doing this from \cite[\S11]{larson-chow-m2}.

To begin with, $\ol{B}$ is a fibre product
\[
  \ol{B}
  \cong
  U_1 \times_X U_2
\]
of the vector bundles $U_1 \defeq V_3(1)/(\GL_2 \times G)$ and $U_2 \defeq (V_1(-1) \boxtimes W_{-2})/(\GL_2 \times G)$ over $X \defeq \B(\GL_2\times G)$. Hence, the Chow ring of $\ol{B}$ is isomorphic to $\CH^*(X)\cong\CH^*(\B\GL_2)\otimes\CH^*(\B G)$ (Chow Künneth holds here by \cite[Lemma~2.12]{totaro-group-cohomology-and-algebraic-cycles}), and so by Lemma~\ref{lem:norm-map-and-presentations}, we have:
\begin{equation}\label{eq:chow-BG}
  \CH^*(\ol B) \cong
  \CH^*(\B\GL_2)\otimes\CH^*(\B G) \cong
  \Z[\alpha_1,\alpha_2,\beta_1,\beta_2,\gamma]/(2\gamma,\gamma^2+\beta_1\gamma).
\end{equation}
The open substack $B \subset \ol{B}$ is the locus parameterizing nonzero polynomials $\lin_1 \cdot \lin_2 \cdot H$ with no triple roots.
In particular, the Chow ring of $B$ is a further quotient of this polynomial ring, as is the Chow ring of the open substack $B \sm Z \subset B$ for the various other closed substacks $Z = \pi\I(\Delta_\star) \subset B$ of interest.
As a first step, we excise the zero sections $U_1 \times 0$ and $0 \times U_2$; this is an Euler class computation, and is carried out in \cite[\S11.3]{larson-chow-m2}.

We may then pass to the projectivizations
\[
  \Pj U_1
  \times_{X}
  \Pj U_2
  \cong
  (\Pj V_3 \times \Pj(V_1 \boxtimes W_{-2})) / (\GL_2 \times G),
\]
where we have used the invariance of projectivization under twisting by line bundles to replace $V_3(1)$ and $V_1(-1)$, by $V_3$ and $V_1$, respectively.  Moreover, each of the closed substacks $Z$ of interest is invariant under scaling in $U_1$ and $U_2$, and hence we obtain an induced closed substack
\[
  \Pj Z
  \hto
  \Pj U_1 \times_X \Pj U_2.
\]
We have a pullback square and thus a commutative square of Chow rings
\[
  \begin{tikzcd}
    Z \sm 0 \ar[r, ""] \ar[d, ""'] \pb & (U_1 \sm 0) \times_X (U_2 \sm 0) \ar[d, ""] \\
    \Pj(Z) \ar[r, ""] & \Pj U_1 \times_X \Pj U_2
  \end{tikzcd}
  \qquad
  \begin{tikzcd}
    \CH^*(Z \sm 0) \ar[r, ""] \ar[from=d, ""'] &
    \CH^*\pbig{(U_1 \sm 0) \times_X (U_2 \sm 0)} \ar[from=d, ""] \\
    \CH^*(\Pj Z) \ar[r, ""] & \CH^*(\Pj U_1 \times_X \Pj U_2).
  \end{tikzcd}
\]
The homomorphism $\CH^*(\Pj Z) \to \CH^*(Z \sm 0)$ is surjective (since $Z \sm 0$ is an open substack of a vector bundle over $\Pj Z$).  Hence to compute the image of $\CH^*(Z \sm 0) \to \CH^*\pbig{(U_1 \sm 0) \times_X (U_2 \sm 0)}$, we may first compute the image of $\CH^*(\Pj Z) \to \CH^*(\Pj U_1 \times_X \Pj U_2)$ (for which \cite{larson-chow-m2} provides tools recalled in \S\ref{subsec:excision-preparations} below), and then pull back the resulting ideal under $(U_1 \sm 0) \times_X (U_2 \sm 0) \to \Pj U_1 \times_X \Pj U_2$.
This last step is done by means of the projective bundle formula, as in Lemma~\ref{lem:proj-pb-formula} below.

The above procedure is carried out in \cite{larson-chow-m2} with $Z \subset \ol{B}$ the triple root locus in order to compute the Chow ring of $B$ as a quotient of $\Z[\alpha_1,\alpha_2,\beta_1,\beta_2,\gamma]$.

Finally, \cite{larson-chow-m2} gives explicit formulas, which we recall in Lemma~\ref{lem:alpha-to-lambda-sub}, allowing us to re-express $\CH^*(B)$ as generated by the pullbacks of $\gamma$ from $\CH^*(\B G)$ and $\delta_1, \lambda_1, \lambda_2$ from $\CH^*(\Mb_2)$.

\subsection{Preparations for the excision computations}\label{subsec:excision-preparations}
We collect here some notation and facts that we will use repeatedly in the computations that follow.
Recall $V_r = \Sym^r(V^*)$ is the $r^{th}$ symmetric product of the dual of the standard representation $V$ of $\GL_2$.

We will adopt the convention in this section that for any fibre product $Y = \Pj E_{1} \times_X \cdots \times_X \Pj E_{k}$ of projective bundles over some base $X$, we write
\[
  x_i \in \CH^*(Y)
\]
for the pullback under the projection \( \pi_i \colon Y \to \Pj E_i \) of the hyperplane class on $\Pj E_i$.
Recall from the end of \S\ref{subsec:conventions} that \( \CH^*(Y) \) is generated as a \( \CH^*(X) \)-algebra by \( x_1,\ldots,x_k \).

Following \cite[Definition~4.1]{larson-chow-m2}, we write
\[
  s_r^j \in \CH^j_{\GL_2}(\Pj V_r),
\]
for $0 \le j \le r$, for the pushforward of $x_1 x_2 \cdots x_j$ under the multiplication map $(\Pj V_1)^j \times \Pj V_{r-j} \rightarrow \Pj V_r$.
Moreover, given a product $Y = \Pj V_{r_1} \times \cdots \times \Pj V_{r_k}$, we write
\[
  u_i^j \in \CH_{\GL_2}^j(Y),
\]
for the pullback along \( \pi_i \colon Y \to \Pj V_{r_i} \) of \( s_{r_i}^j \in \CH^j_{\GL_2}(\Pj V_{r_i}) \).
We recall the following results that we shall need from \cite{larson-chow-m2}.
\begin{lem}[{\cite[Lemma~4.2]{larson-chow-m2}}]\label{lem:uij-formulas}
  The classes \( u_i^j \in \CH^j_{\GL_2}( \Pj V_{r_1} \times \cdots \times \Pj V_{r_k}) \) are given by the recursive formula:
  \begin{equation}
    u_i^0 = 1
    \qquad
    \text{and}
    \qquad
    u_i^{j+1} =
    (x_i - j \alpha_1) \cdot u_i^j +
    j (r_i + 1 - j) \alpha_2 \cdot u_i^{j-1}
  \end{equation}
\end{lem}
\begin{lem}[{\cite[Lemma~4.4]{larson-chow-m2}}]\label{lem:pushforward-sij}
The pushforward in \( \CH_{\GL_2}^* \) along the multiplication map
\[
    \mul \colon \Pj V_a \times \Pj V_b \to \Pj V_{a+b}
\]
is given by
\[
    u_1^\alpha \cdot u_2^\beta
    \mapsto
    {a - \alpha + b - \beta \choose a - \alpha} \cdot u_1^{\alpha + \beta}.
\]
\end{lem}
\begin{lem}[{\cite[Lemma~4.7]{larson-chow-m2}}]\label{lem:squaring-map}
The pushforward in \( \CH_{\GL_2}^* \) along the squaring map
\[
  \sq \colon
  \Pj V_1 \to \Pj V_2
\]
is given by
\[
  u_1^0 \mapsto 2 u_1^1 - 2 \alpha_1
  \quad
  \text{and}
  \quad
  u_1^1 \mapsto u_1^2 - 2 \alpha_2.
\]
\end{lem}
\begin{lem}[{\cite[Lemma~4.5]{larson-chow-m2}}]\label{lem:diagonal-class}
  The pushforward of the fundamental class along the diagonal map \( \Pj V_1 \to \Pj V_1 \times \Pj V_1 \) --- i.e., the class of the diagonal $\Delta \subset \Pj V_1 \times \Pj V_1$ --- is given by
  \[
    [\Delta]_{\GL_2} =
    x_1 + x_2 - \alpha_1
    \in \CH_{\GL_2}^*(\Pj V_1 \times \Pj V_1).
  \]
\end{lem}

We will make frequent use of the following projection map:
\begin{defn}\label{defn:phi-projector}
 Let $I$ be the union of $\Pj(V_1 \boxtimes (L_{-2} \times 0))$ and $\Pj(V_1 \boxtimes (0 \times L_{-2}))$ inside $\Pj(V_1 \boxtimes W_{-2})$.
  We denote by
  \[
    \phi \colon \Pj(V_1 \boxtimes W_{-2})\sm I
    \to
    \Pj V_1 \times \Pj V_1
  \]
  the map induced by the linear projections \[V_1 \boxtimes L_{-2} \xot{\id \boxtimes \pi_1} V_1 \boxtimes (L_{-2} \times L_{-2}) \tox{\id \boxtimes \pi_2} V_1 \boxtimes L_{-2}.\]
\end{defn}

\noindent On the codomain of \( \phi \), the line bundle $\oo(1)$ on each factor is a \( \GL_2 \times \Gm \times \Gm \)-equivariant (but not \( (\GL_2 \times G ) \)-equivariant) vector bundle, and hence we have classes \( x_1, x_2 \in \CH^1_{\GL_2 \times \Gm \times \Gm}(\Pj V_1 \times \Pj V_1) \).

\begin{lem}\label{lem:hyperplane-t2-twist}
  For $i = 1,2$, we have:
  \[
    \phi^* x_i = x_1 - 2 t_i \in \CH^*_{\GL_2 \times \Gm \times \Gm}\pbig{\Pj(V_1 \boxtimes W_{-2}) \sm I}
  \]
\end{lem}

\begin{proof}
  In general, given vector bundles $E_1,E_2 \to X$ over a base $X$ and a bundle map \( f \colon E_1 \to E_2 \), and writing $\Pj f \colon \Pj E_1 \sm \Pj(\ker f ) \to \Pj E_2$ for the induced map, we have \( (\Pj f)^*x_1 = x_1 \).
  For each $i = 1,2$, we now apply this to the map
  \[
    [\Pj(V_1 \boxtimes W_{-2}) \sm I] / H
    \toi
    \sqbig{\Pj\pbig{(V_1 \boxtimes L_{-2}) \oplus (V_1 \boxtimes L_{-2})} \sm I} / H
    \tox{\Pj \pi_i}
    \Pj(V_1 \boxtimes L_{-2}) / H
    \toi
    \Pj V_1 / H
  \]
  of projective bundles over \( \B H \), where \( H = \GL_2 \times \Gm \times \Gm \), and use that the last of these maps pulls back $x_1$ to $x_1 - 2t_i$ by the last sentence of \S\ref{subsec:conventions}.
\end{proof}

Finally, we record the formulas mentioned at the end of \S\ref{subsec:biell-excision-strategy}:
\begin{lem}\label{lem:proj-pb-formula}
  The pullback in either \( \CH_{\GL_2 \times G} \) or \( \CH_{\GL_2 \times \Gm \times \Gm} \) along the projectivization map
  \[
    [V_3(1) \sm 0] \times [\pbig{V_1(-1) \boxtimes W_{-2}} \sm 0]
    \to
    \Pj V_3 \times \Pj(V_1 \boxtimes W_{-2})
  \]
  is given by \( x_1 \mapsto \alpha_1 \) and \( x_2 \mapsto -\alpha_1 \).
\end{lem}
\begin{proof}
  In general, for a vector bundle \( \mathcal{E} \), the hyperplane class on \( \Pj(\mathcal{E}) \) vanishes upon pulling back under the projectivization map \( \mathcal{E} \sm 0 \to \Pj(\mathcal{E}) \).
  Thus, pulling back under the composite \( V_3(1) \sm 0 \to \Pj\pbig{V_3(1)} \toi \Pj(V_3) \), we obtain \( x_1 \mapsto x_1 + \ch_1(\det V) \mapsto 0 + \ch_1(\det V) = \alpha_1 \), using the twisting formula at the end of \S\ref{subsec:conventions}.
  The computation on the \( \pbig{V_1(-1) \boxtimes W_{-2}} \) factor is similar.
\end{proof}

\begin{lem}[{\cite[Section~11.6]{larson-chow-m2}}]\label{lem:alpha-to-lambda-sub}
  The generators \(\alpha_1\), \(\alpha_2\), \(\beta_1\), and \(\beta_2\) for $\CH^*(B)$ are expressed in terms of \( \gamma \in \CH^*(\B G) \) and the pullbacks of \(\delta_1, \lambda_1, \lambda_2 \in \CH^*(\Mb_2) \) along \( \pi \colon B \to \Mb_2 \) as follows:
  \begin{align*}
    \alpha_1 & = -2\lambda_1 + \delta_1 + \gamma &
    \beta_1 & =  3\lambda_1 - 2\delta_1\\
    \alpha_2 & =  2\delta_1\lambda_1  -2\lambda_1\gamma - 8 \lambda_2 &
    \beta_2 & =  2\lambda_1^2 - 3\delta_1\lambda_1 + \delta_1^2 + \lambda_1\gamma + \gamma^2 + \lambda_2
  \end{align*}
\end{lem}

\subsection{The Chow ring of \texorpdfstring{\boldmath \( B \sm \pi^{-1}(\Delta_{000}) \)}{B \ π⁻¹(Δ₀₀₀)}}\label{subsec:b-excisions-d000}
By Lemma~\ref{lem:loci-in-b}, $\pi\I(\Delta_{000}) \subset B$ is the locus of
\[
    (H,\lin_1,\lin_2) \in
   B \subset \sqbig { V_3(1) \times \pbig{V_1(-1) \boxtimes W_{-2}} } / (\GL_2 \times G)
\]
where $H$ is of the form $\lin_1 R^2$ or $\lin_2 R^2$ for some linear polynomial $R$.
Let
\[
\wt{B} \subset \sqbig { V_3(1) \times \pbig{V_1(-1) \boxtimes W_{-2}} } / (\GL_2 \times \Gm \times \Gm)
\]
be the complement of the locus where $H \cdot \lin_1 \cdot \lin_2$ has a triple root, and let $\wt{\pi\I(\Delta_{000})} \subset \wt{B}$ denote the locus consisting of $(H,\lin_1,\lin_2)$ with $H = \lin_1 R^2$ for some \( R \).
We have a commutative square
\begin{equation}\label{eq:norm-pb-square}
  \begin{tikzcd}
    \wt{\pi\I(\Delta_{000})}
    \ar[r, hookrightarrow] \ar[d, "\sim"' sloped] &
    \wt{B}
    \ar[d, ""] \\
    \pi\I(\Delta_{000}) \ar[r, hookrightarrow] &
    B
  \end{tikzcd}
\end{equation}
where the arrow on the left is an isomorphism.
Thus, the pushforward along the bottom arrow is the composite of the pushforwards along the top and right arrow.
The second of these is given by the norm map from \( \GL_2 \times \Gm \times \Gm \) to \( \GL_2 \times G \).

We now pass to the projectivizations
\[
  \sqbig{ \Pj V_3 \times \Pj(V_1 \boxtimes W_{-2}) } /
  (\GL_2 \times \Gm \times \Gm)
\]
as in \S\ref{subsec:biell-excision-strategy}.
Let $\theta$ be the multiplication map
\[
  \theta \colon \Pj V_1 \times \Pj V_1 \times \Pj V_1\rightarrow \Pj V_3 \times \Pj V_1 \times \Pj V_1
\]
defined by $(f,g,h) \mapsto (f^2g, g, h)$.  We have a pullback square
\begin{equation*}
    \begin{tikzcd}
    \Pj\pbig{\wt{\pi\I(\Delta_{000})}} \ar[r, hookrightarrow, "i"] \ar[d, ""'] \pb &
    { [ \Pj V_3 \times (\Pj(V_1 \boxtimes W_{-2}) \sm I) ] }
    / ( \GL_2 \times \Gm \times \Gm )
    \ar[d, "\id \times \phi"] \\
    {[ \Pj V_1 \times \Pj V_1 \times \Pj V_1 ]}
    / ( \GL_2 \times \Gm \times \Gm )
    \ar[r, hookrightarrow, "\theta"] &
    ( \Pj V_3 \times \Pj V_1 \times \Pj V_1 )
    / ( \GL_2 \times \Gm \times \Gm )
    \end{tikzcd}
\end{equation*}
and hence a commutative square of Chow rings
\begin{equation}\label{eq:id-times-phi-pb-square}
\begin{tikzcd}
    \CH^*\pBig{\Pj\pbig{\wt{\pi\I(\Delta_{000})}}} \ar[r, "i_*"] \ar[from=d, ""'] &
    \CH^*_{\GL_2 \times \Gm \times \Gm}\pbig{ \Pj V_3 \times (\Pj(V_1 \boxtimes W_{-2}) \sm I)
    }
    \ar[from=d, "(\id \times \phi)^*"] \\
    \CH^*_{\GL_2 \times \Gm \times \Gm}(\Pj V_1 \times \Pj V_1 \times \Pj V_1)
    \ar[r, "\theta_*"] &
    \CH^*_{\GL_2 \times \Gm \times \Gm}(\Pj V_3 \times \Pj V_1 \times \Pj V_1)
    .
  \end{tikzcd}
\end{equation}
The vertical maps in the top square are $\Gm$-bundles, and hence the vertical arrows in the bottom square are surjective.
Thus our desired image $\im(i_*)$ can be computed as the image of the composite of $\theta_*$ and $(\id \times \phi)^*$.

Pullback along $\theta$ is not surjective on Chow groups; instead, the fundamental class and $x_1$ generate $\CH^*_{\GL_2 \times \Gm \times \Gm}(\Pj V_1 \times \Pj V_1 \times \Pj V_1)$ as a $\CH^*_{\GL_2 \times \Gm \times \Gm}(\Pj V_3 \times \Pj V_1 \times \Pj V_1)$ module.
Thus, by \S\ref{subsec:push-pull}, the image of $\theta_*$ is generated as an ideal by $\theta_*(1)$ and $\theta_*(x_1)$.

The map $\theta$ factors through a diagonal map on the middle $\Pj V_1$ factor
\begin{equation}\label{eq:d000-diag}
  \Pj V_1 \times \Pj V_1 \times \Pj V_1
  \tox{\id \times \Delta \times \id}
  \Pj V_1 \times (\Pj V_1 \times \Pj V_1) \times \Pj V_1,
\end{equation}
followed by a squaring map on the first factor
\begin{equation}\label{eq:d000-square}
  \Pj V_1 \times \Pj V_1 \times \Pj V_1 \times \Pj V_1
  \tox{\sq \times \id \times \id \times \id}
  \Pj V_2 \times \Pj V_1 \times \Pj V_1 \times \Pj V_1,
\end{equation}
and finally a multiplication map on the first two factors
\begin{equation}\label{eq:d000-mult}
  (\Pj V_2 \times \Pj V_1) \times \Pj V_1 \times \Pj V_1
  \tox{\mul \times \id \times \id}
  \Pj V_3 \times \Pj V_1 \times \Pj V_1.
\end{equation}
From Lemma~\ref{lem:diagonal-class}, the pushforward of the fundamental class along \eqref{eq:d000-diag} is
\( [\Delta_{23}] = x_2 + x_3 -\alpha_1, \)
and since pullback along \eqref{eq:d000-diag} takes \( x_1 \) to \( x_1 \), the pushforward of \( x_1 \) along \eqref{eq:d000-diag} is thus \( x_1 [\Delta_{23}] \) by \S\ref{subsec:push-pull}.
Next, by Lemma~\ref{lem:squaring-map}, and since pullback along \eqref{eq:d000-square} takes \( \Delta_{23} \) to \( \Delta_{23} \), we have that pushing forward under \eqref{eq:d000-square} gives
\begin{align*}
  [\Delta_{23}] & =
  u_1^0 \cdot [\Delta_{23}]
  \\ & \mapsto
  \pbig{2 u_1^1 - 2\alpha_1} [\Delta_{23}]
  =
  \pbig{2 u_1^1 - 2\alpha_1}(x_2 + x_3 - \alpha_1)\\
  x_1 [\Delta_{23}] & =
  u_1^1 \cdot [\Delta_{23}]
  \\ & \mapsto
  \pbig{u_1^2 - 2 \alpha_2} [\Delta_{23}]
  =
  \pbig{u_1^2 - 2 \alpha_2} (x_2 + x_3 - \alpha_1).
\end{align*}
To pushforward along \eqref{eq:d000-mult}, we use Lemma \ref{lem:pushforward-sij}, and the fact that pullback along \eqref{eq:d000-mult} takes \( x_2 \) to \( x_3 \):
\begin{align*}
  (2u_1^1 - 2\alpha_1)(x_2 + x_3 - \alpha_1)
  & =
  (2u_1^1 - 2 u_1^0 \alpha_1 )\pbig{u_2^1 + u_2^0(x_3 - \alpha_1)}
  \\ & \mapsto
  2u_1^2 + 4u_1^1(x_2 - \alpha_1) - 2\alpha_1(u_1^1 + 3u_1^0(x_2 - \alpha_1))
  \\
  (u_1^2 - 2 \alpha_2)(x_2 + x_3 - \alpha_1)
  & =
  (u_1^2 - 2 u_1^0 \alpha_2)\pbig{u_2^1 + u_2^0(x_3 - \alpha_1)}
  \\ & \mapsto
  u_1^3 + u_1^2(x_2 - \alpha_1) - 2\alpha_2 \pbig{u_1^1 + 3 u_1^0 (x_2 - \alpha_1)}.
\end{align*}
We have thus computed the image of \( \theta_* \) as the ideal generated by the above two elements.
To pull these back along \( \id \times \phi \), we substitute \( x_2 \mapsto x_2 - 2t_1 \) as per Lemma~\ref{lem:hyperplane-t2-twist}.
Then, to pull back from the projectivization, we substitute \( x_1 \mapsto \alpha_1 \) and \( x_2 \mapsto -\alpha_1 \) as per Lemma~\ref{lem:proj-pb-formula}.
At the same time, we use Lemma~\ref{lem:uij-formulas} to plug in for the various $u_1^i$:
\begin{equation} \label{eq:uij-on-pv3-values}
\begin{aligned}
u_1^0 &= 1 \hspace{80pt} &
u_1^2 &= (x_1 - \alpha_1) u_1^1 + 3 \alpha_2 u_1^0
\mapsto 3 \alpha_2 \\
u_1^1 &= x_1 \mapsto \alpha_1 &
u_1^3 &= (x_1 - 2 \alpha_1) u_1^2 + 4 \alpha_2 u_1^1
\mapsto \alpha_1 \alpha_2.
\end{aligned}
\end{equation}
This yields
\begin{align*}
  2u_1^2 + 4u_1^1(x_2 - \alpha_1) - 2\alpha_1 \pbig{u_1^1 + 3u_1^0(x_2 - \alpha_1)}
  & \mapsto 6\alpha_2 + 4\alpha_1(-2\alpha_1 - 2t_1) - 2\alpha_1 \pbig{\alpha_1 + 3(-2\alpha_1 - 2t_1)} \\
  & = 6\alpha_2 + 4\alpha_1t_1 +2\alpha_1^2, \\
  u_1^3 + u_1^2(x_2 - \alpha_1) - 2\alpha_2 \pbig{u_1^1 + 3 u_1^0 (x_2 - \alpha_1)}
  & \mapsto \alpha_1\alpha_2 + 3\alpha_2(-2\alpha_1 -2t_1) -2\alpha_2(\alpha_1 + 3(-2\alpha_1 - 2t_1))\\
  & = 5\alpha_1\alpha_2 + 6\alpha_2t_1.
\end{align*}
Finally, we need to push forward the ideal generated by the above elements along the right-hand vertical arrow in \eqref{eq:norm-pb-square}, which is a base change of the norm map \( \pi_* \colon \B(\Gm \times \Gm) \rightarrow \B G \).
Referring to Lemma~\ref{lem:norm-map-and-presentations}, we need to push forward the generators and their products with \( t_2 \).
We obtain:
\begin{equation}\label{eq:B-minus-d000-norm-push}
  \begin{split}
    6\alpha_2 + 4\alpha_1t_1 +2\alpha_1^2 & \mapsto 
    12\alpha_2 + 4\alpha_1^2 + 4\alpha_1(\beta_1 + \gamma)\\
    t_2 (6\alpha_2 + 4\alpha_1t_1 +2\alpha_1^2) & \mapsto 
    (6\alpha_2 + 2\alpha_1^2)(\beta_1 +\gamma) +4\alpha_1(2\beta_2)\\
    5\alpha_1\alpha_2 + 6\alpha_2t_1 & \mapsto 
    10\alpha_1\alpha_2 + 6\alpha_2(\beta_1 +\gamma)\\
    t_2 (5\alpha_1\alpha_2 + 6\alpha_2t_1) & \mapsto 
    5\alpha_1\alpha_2(\beta_1 +\gamma) + 6\alpha_2(2\beta_2)
    . 
  \end{split}
\end{equation}
The last three relations are already equivalent to 0 in $\CH^*(B)$, and using Lemma~\ref{lem:alpha-to-lambda-sub}, the first relation is equal to \( - 16\lambda_1^2 + 20\lambda_1\delta_1 \), giving the relation for \( \CH^*\pbig{B \sm \pi\I(\Delta_{000})} \) appearing in Proposition~\ref{propn:biell-summary}.

\subsection{The Chow ring of \texorpdfstring{\boldmath \( B \sm \pi^{-1}(\Delta_{00}) \)}{B \ π⁻¹(Δ₀₀)}; first part}
By Lemma~\ref{lem:loci-in-b}, \( \pi^{-1}(\Delta_{00}) \subset B \) is the locus of \( (H,\lin_1,\lin_2) \) where the cubic \( H \) has a double root, or where both linear factors \( \lin_i \) divide \( H \).

We excise the first of these here, and carry out the second -- which is somewhat more involved, but turns out not to produce any new relations -- in \S\ref{subsec:b-excisions-d00-part2} below.
The present computation is similar to the one in the previous section, but easier, since we do not need to break the $\Z/2\Z$ symmetry and pass from $G$ to $\Gm \times \Gm$.

As before, we pass to the product of projectivizations
\[
  \sqbig{ \Pj V_3 \times \Pj(V_1 \boxtimes W_{-2}) } / (\GL_2 \times G).
\]
The locus we want to excise can be described as the image of the map
\[
  \theta \colon \Pj V_1 \times \Pj V_1 \times \Pj(V_1 \boxtimes W_{-2}) \rightarrow \Pj V_3 \times \Pj(V_1 \boxtimes W_{-2})
\]
given by $(f,g,h_1,h_2) \rightarrow (f^2g,h_1,h_2)$.
As a $\CH^*(\Pj V_3)$ module, $\CH^*(\Pj V_1 \times \Pj V_1)$ is generated by the fundamental class $1 = u_1^0$ and the hyperplane class $x_1 = u_1^1$ on the first $ \Pj V_1$ factor.
Thus, by \S\ref{subsec:push-pull}, the images of these classes will generate \( \im(\theta_*) \) as an ideal.

The map $\theta$ factors as a squaring map
\begin{equation}\label{eq:d00-part1-square}
  \Pj V_1 \times \Pj V_1 \times \Pj(V_1 \boxtimes W_{-2})
  \tox{\sq \times \id \times \id}
  \Pj V_2 \times \Pj V_1 \times \Pj(V_1 \boxtimes W_{-2})
\end{equation}
followed by a multiplication map
\begin{equation}\label{eq:d00-part1-mult}
  (\Pj V_2 \times \Pj V_1) \times \Pj(V_1 \boxtimes W_{-2})
  \tox{\mul \times \id}
  \Pj V_3 \times \Pj(V_1 \boxtimes W_{-2}).
\end{equation}
Using Lemmas~\ref{lem:squaring-map}~and~\ref{lem:pushforward-sij}, pushing forward along \eqref{eq:d00-part1-square} and \eqref{eq:d00-part1-mult} gives
\begin{align*}
  u_1^0 & \mapsto
  2 u_1^1 - 2\alpha_1 =
  (2 u_1^1 - 2\alpha_1) \cdot u_2^0 \mapsto
  4u_1^1 - 6\alpha_1\\
  u_1^1 & \mapsto
  u_1^2 - 2\alpha_2 =
  (u_1^2 - 2\alpha_2) \cdot u_2^0 \mapsto
  u_1^2 - 6\alpha_2.
\end{align*}
Pulling back along the projectivization map using Lemma~\ref{lem:proj-pb-formula} and plugging in the values of the \( u_1^i \) using \eqref{eq:uij-on-pv3-values}, we obtain relations \( -2\alpha_1 \) and \( -3\alpha_2 \).
The latter element is already 0 in \( \CH^*(B) \), and by Lemma~\ref{lem:alpha-to-lambda-sub}, the first is equal to \( 4\lambda_1 - 2\delta_1 \), giving the relation for $B \sm \pi\I(\Delta_{00})$ appearing in Proposition~\ref{propn:biell-summary}.

\subsection{The Chow ring of \texorpdfstring{\boldmath \( B \sm \pi^{-1}(\Delta_{00}) \)}{B \ π⁻¹(Δ₀₀)}; second part}\label{subsec:b-excisions-d00-part2}
We now excise the locus where both linear factors divide the cubic.
After passing to projectivizations, this is the image of the map
\begin{equation}\label{eq:b-excisions-d00-part2-map}
  \Pj V_1 \times [ \Pj(V_1 \boxtimes W_{-2}) \sm I ]
  \to
  \Pj V_3 \times [ \Pj(V_1 \boxtimes W_{-2}) \sm I ]
\end{equation}
given by $(f,g,h) \mapsto (fgh,g,h)$.
Since the pullback under this map is surjective on Chow rings, it suffices by \S\ref{subsec:push-pull} to push forward the fundamental class.

We define the map
\[
  \psi \colon\Pj(V_1 \boxtimes W_{-2}) \sm I \to \Pj V_2
\]
by \( (f,g) \mapsto f \cdot g \).
In other words, \( \psi = \mul \circ \phi \), with \( \phi \) as in Definition~\ref{defn:phi-projector}.
We may then factor \eqref{eq:b-excisions-d00-part2-map} as the map
\begin{equation}\label{eq:b-minus-d00-first-factor}
  \Pj V_1 \times [ \Pj(V_1 \boxtimes W_{-2}) \sm I ]
  \tox{\id \times \br{\psi,\id}}
  \Pj V_1 \times \pbig{\Pj V_2 \times [ \Pj(V_1 \boxtimes W_{-2}) \sm I]},
\end{equation}
given by \( (f,g,h) \mapsto (f, g \cdot h, g, h) \), followed by the multiplication map
\begin{equation}\label{eq:b-minus-d00-second-factor}
  (\Pj V_1 \times \Pj V_2) \times [\Pj(V_1 \boxtimes W_{-2}) \sm I]
  \tox{\mul \times \id}
  \Pj V_3 \times [\Pj(V_1 \boxtimes W_{-2}) \sm I].
\end{equation}
To begin with, we compute the image of the pushforward of the fundamental class under
\[
  \br{\psi,\id} \colon [\Pj(V_1 \boxtimes W_{-2}) \sm I]
  \to
  \Pj V_2 \times [\Pj(V_1 \boxtimes W_{-2}) \sm I].
\]
This is the class
\[
  \Psi \in \CH_{\GL_2 \times G}^*\pbig{\Pj V_2 \times [\Pj(V_1 \boxtimes W_{-2}) \sm I]}
\]
of the graph of \( \psi \).
\begin{lem}\label{lem:class-of-graph}
  We have
  \[
    \Psi = 2\alpha_2 + 8\beta_2 + 2\alpha_1^2 + 4\alpha_1\beta_1 -4\alpha_1x_2 -4\beta_1x_2 -3\alpha_1x_1
    -2\beta_1x_1 + 2x_1x_2 + 2x_2^2 + x_1^2.
  \]
\end{lem}
\begin{proof}
  Since the graph of \( \psi \) has codimension 2, we may write
  \begin{align*}
    \Psi &= a\alpha_2 + b\beta_2 + c\alpha_1^2 + d\beta_1^2 + e\alpha_1\beta_1 + f\alpha_1x_2 + g\beta_1x_2 + h\alpha_1x_1\\
    &\phantom{= .\!}
    + i\beta_1x_1 + jx_1x_2 + kx_2^2 + lx_1^2 + m\alpha_1\gamma + n\beta_1\gamma + p\gamma x_2 + q\gamma x_1 + r \gamma^2
  \end{align*}
  for some undetermined coefficients \( a,\ldots,r \in \Z \).  Using the relation $\gamma^2 + \beta_1 \gamma = 0$, we assume $r = 0$.
  We will determine most of these coefficients by pulling back to \( \CH^*_{\GL_2 \times \Gm \times \Gm}\pbig{\Pj V_2 \times \Pj(V_1 \boxtimes W_{-2})} \).
  By Lemma~\ref{lem:norm-map-and-presentations}, the pullback $\Psi_{\GL_2 \times \Gm \times \Gm}$ has class
  \begin{equation}\label{eq:omega-gl2-gm-gm-undetermined}
    \begin{split}
      \Psi_{\GL_2 \times \Gm \times \Gm} & =
      a\alpha_2 + bt_1t_2 + c\alpha_1^2 + d(t_1 + t_2)^2 + e\alpha_1(t_1 + t_2) + f\alpha_1x_2 + g(t_1 + t_2)x_2\\
      & \phantom{= .\!}
      + h\alpha_1x_1 + i(t_1 + t_2)x_1 + jx_1x_2 + kx_2^2 + lx_1^2.
    \end{split}
  \end{equation}
  Now, using the commutative squares of Chow rings induced by the pullback square
  \[
    \begin{tikzcd}
      {\Pj(V_1 \boxtimes W_{-2}) \sm I}
       \ar[rr, "\br{\psi,\id}"]
      \ar[d, "\phi"]
      \pb
      &
      &
      \Pj V_2 \times [\Pj(V_1 \boxtimes W_{-2}) \sm I]
      \ar[d, "\id \times \phi"]
      \\
      \Pj V_1 \times \Pj V_1
      \ar[r, "\Delta"]
      &
      (\Pj V_1 \times \Pj V_1) \times (\Pj V_1 \times \Pj V_1)
      \ar[r, "\mul \times \id"]
      &
      \Pj V_2 \times (\Pj V_1 \times \Pj V_1),
    \end{tikzcd}
  \]
  $\Psi_{\GL_2 \times \Gm \times \Gm}$ can be described as the pullback under
  \begin{equation}\label{eq:b-minus-d00-part-2-idphi}
    \id \times \phi \colon
    \Pj V_2 \times [\Pj(V_1 \boxtimes W_{-2}) \sm I]
    \to
    \Pj V_2 \times (\Pj V_1 \times \Pj V_1)
  \end{equation}
  of the pushforward of the bi-diagonal \( [\Delta_{13}] \cdot [\Delta_{24}] \) under
  \[
    (\Pj V_1 \times \Pj V_1)
    \times
    (\Pj V_1 \times \Pj V_1)
    \tox{\mul \times \id}
    \Pj V_2
    \times
    (\Pj V_1 \times \Pj V_1).
  \]
  By Lemma~\ref{lem:diagonal-class}, we have
  \begin{align*}
    [\Delta_{13}] \cdot [\Delta_{24}]
    & = x_1x_2 + x_1x_4 + x_2x_3 +x_3x_4 -\alpha_1(x_1 + x_2 + x_3 + x_4) + \alpha_1^2\\
    & = x_1x_2 + x_1(x_4-\alpha_1) + x_2(x_3-\alpha_1) +x_3x_4 -\alpha_1(x_3 + x_4) + \alpha_1^2\\
    & = u_1^1 \cdot u_2^1 + (u_1^1 \cdot u_2^0)(x_4-\alpha_1)
    + (u_1^0 \cdot u_2^1)(x_3 - \alpha_1) + (u_1^0 \cdot u_2^0)(x_3x_4 -\alpha_1(x_3 + x_4) + \alpha_1^2).
  \end{align*}
  Pushing this forward along \( \mul \times \id \) using Lemma~\ref{lem:pushforward-sij} gives
  \begin{align*}
    [\Delta_{24}] \cdot [\Delta_{13}]
    & \mapsto u_1^2 + u_1^1(x_3-\alpha_1) + u_1^1(x_2-\alpha_1) + 2u_1^0(x_2x_3 -\alpha_1(x_2 + x_3) + \alpha_1^2)\\
    & = u_1^2 + u_1^1(x_2 + x_3 -2\alpha_1) + 2u_1^0(x_2-\alpha_1)(x_3 -\alpha_1).
  \end{align*}
  Finally, pulling back along \eqref{eq:b-minus-d00-part-2-idphi} by making the substitutions $x_2 \mapsto x_2 - 2t_1$ and $x_3 \mapsto x_2 - 2 t_2$ as per Lemma~\ref{lem:hyperplane-t2-twist}, and using the values
  \begin{equation}\label{eq:uij-on-pv2-values}
    u_1^0 = 1,
    \qquad
    u_1^1 = x_1,
    \quad
    \text{and}
    \quad
    u_1^2 = (x_1 - \alpha_1) x_1 + 2\alpha_2
  \end{equation}
  from Lemma~\ref{lem:uij-formulas}, we obtain
  \begin{align*}
    \Psi_{\GL_2 \times \Gm \times \Gm} & =
    (x_1-\alpha_1)x_1 +2\alpha_2 +x_1(2x_2 -2\alpha_1 - 2t_1 -2t_2) +2(x_2-\alpha_1 - 2t_1)(x_2-\alpha_1-2t_2)\\
    & = x_1^2 -3\alpha_1 x_1 + 2\alpha_2 + 2x_1x_2 -2x_1(t_1 + t_2) + 2x_2^2 -4\alpha_1 x_2 \\
    & \phantom{= .\!} + 2\alpha_1^2 -4x_2(t_1 + t_2) +4\alpha_1(t_1 + t_2) +8t_1t_2.
  \end{align*}
  Comparing this with \eqref{eq:omega-gl2-gm-gm-undetermined}, we obtain $a = 2$, $b = 8$, $c = 2$, $d = 0$, $e = 4$, $f = -4$, $g = -4$, $h = -3$, $i = -2$, $j = 2$, $k = 2$, and $l = 1$.
  In other words,
  \begin{align*}
    \Psi & =
    2\alpha_2 + 8\beta_2 + 2\alpha_1^2 + 4\alpha_1\beta_1 -4\alpha_1 x_2 -4\beta_1 x_2 -3\alpha_1 x_1\\
    & \phantom{= .\!}
    -2\beta_1 x_1 + 2 x_1 x_2 + 2 x_2^2 + x_1^2 + m\alpha_1\gamma + n\beta_1\gamma + p\gamma x_2 + q\gamma x_1.
  \end{align*}

  To determine \( m \), \( n \), \( p \), \( q \), we pull \( \Psi \) back to \( \CH^*_{\GL_2 \times \Z/2\Z} \).
  Under this pullback, the representation \( W_1 \) of \( G \) pulls back to the trivial representation plus the sign representation of \( \Z/2\Z \), and hence we have \( \beta_1 \mapsto \gamma \) and \( \beta_2 \mapsto 0 \). Thus
\begin{equation}\label{eq:omega-gl2-z2-undetermined}
  \begin{split}
    \Psi_{\GL_2 \times \Z / 2\Z}
    &=
    2\alpha_2 + 2\alpha_1^2 + 4\alpha_1\gamma - 4\alpha_1 x_2 - 4 \gamma x_2 - 3\alpha_1x_1\\
    &\phantom{= .\!}
    -2 x_1 \gamma + 2 x_1 x_2 + 2 x_2^2 + x_1^2
    + m \alpha_1 \gamma + n \gamma^2 + p \gamma x_2 + q \gamma x_1
    \\ &=
    2\alpha_2 + 2\alpha_1^2 - 4 \alpha_1 x_2
    - 3 \alpha_1 x_1 + 2 x_1 x_2 + 2 x_2^2 + x_1^2
    \\
    &\phantom{= .\!}
    + m\alpha_1 \gamma + n\gamma^2 + p\gamma x_2 + q \gamma x_1,
    \end{split}
  \end{equation}
  where in the last equality we have used that \( 2\gamma = 0 \).
Consider the pullback of \( \Psi_{\GL_2 \times \Z / 2\Z} \) along the \( \GL_2 \times \Z / 2\Z \)-equivariant inclusion \( f \colon \Pj V_2 \times \Pj(V_1 \boxtimes L_{-2}) \to \Pj V_2 \times [\Pj(V_1 \boxtimes W_{-2}) \sm I] \) induced by the diagonal embedding \( L_{-2} \to W_{-2} \).
  Since \( f^*x_i = x_i \) for \( i = 1,2 \), we have the same expression for \( f^* \Psi_{\GL_2 \times \Z / 2\Z} \) as the one given for \( \Psi_{\GL_2 \times \Z / 2\Z} \) above.

  Note that $\psi \colon \Pj(V_1 \boxtimes W_{-2}) \sm I \to \Pj V_2$ restricts to the squaring map $\sq \colon \Pj V_1 \to \Pj V_2$ on $\Pj(V_1 \boxtimes L_{-2})$, and hence $f^* \Psi_{\GL_2 \times \Z / 2\Z}$ is the class of the graph of \( \sq \).
  Working as above, we may compute this class as the pullback under $\Pj V_2 \times \Pj (V_1 \boxtimes L_{-2}) \toi \Pj V_2 \times \Pj V_1$ of the pushforward of the diagonal $[\Delta]$ under the squaring map
  \[
    \sq \times \id \colon \Pj V_1 \times \Pj V_1 \rightarrow \Pj V_2 \times \Pj V_1.
  \]
  By Lemma~\ref{lem:diagonal-class}, we have
  \(
  [\Delta]
  = x_1 + x_2 - \alpha_1
  = u_1^1 + u_1^0(x_2 - \alpha_1).
  \)
  Using Lemma~\ref{lem:squaring-map} and the values of the \( u_1^i \) from \eqref{eq:uij-on-pv2-values}, we obtain
  \begin{align*}\label{eq: class of graph}
    (\sq \times \id)_*[\Delta] & =
    u_1^2 - 2\alpha_2 + (2 u_1^1 - 2\alpha_1)(x_2 - \alpha_1) \\
    & =
    x_1^2 -3\alpha_1 x_1 + 2 x_1 x_2 - 2\alpha_1 x_2 + 2\alpha_1^2.
  \end{align*}
  Comparing this with \eqref{eq:omega-gl2-z2-undetermined}, we obtain \( m,n,p,q=0 \), as required, since \( x_2^2 - \alpha_1 x_2 + \alpha_2 = 0 \) by the projective bundle formula \eqref{eq:proj-bdl}.
\end{proof}

We return to our task of pushing forward the fundamental class along \eqref{eq:b-minus-d00-first-factor} and \eqref{eq:b-minus-d00-second-factor}.
Writing
\[
  \pi_2 \colon \Pj V_1 \times \pbig{\Pj V_2 \times [ \Pj(V_1 \boxtimes W_{-2}) \sm I]}
  \to
  \Pj V_2 \times [ \Pj(V_1 \boxtimes W_{-2}) \sm I]
\]
for the projection map, the pushforward of the fundamental class along \eqref{eq:b-minus-d00-first-factor} is $\pi_2^*\Psi$.
By Lemma~\ref{lem:class-of-graph}, this is given by
\[
  \pi_2^*\Psi =
  2\alpha_2 + 8\beta_2 + 2\alpha_1^2 + 4\alpha_1\beta_1 -4\alpha_1x_3 -4\beta_1x_3 -3\alpha_1x_2
  -2\beta_1x_2 + 2 x_2 x_3 + 2x_3^2 + x_2^2.
\]
To push this forward along the multiplication map \eqref{eq:b-minus-d00-second-factor}, we first re-express the powers of \( x_2 \) in terms of \( u_2^i \) using that \( u_2^0 = 1 \), \( u_2^1 = x_2 \), and \( u_2^2 = (x_2 - \alpha_1) x_2 + 2\alpha_2 \) as in \eqref{eq:uij-on-pv2-values}.
We obtain
\begin{align*}
    \pi_2^*\Psi = & u_2^2 + u_2^1(2x_3 -2\beta_1 -2\alpha_1) + u_2^0(8\beta_2 + 2\alpha_1^2 + 4\alpha_1\beta_1 -4\alpha_1x_3 -4\beta_1x_3 + 2x_3^2).
\end{align*}
Pushing \( \pi_2^*\Psi = u_1^0 \cdot \pi_2^*\Psi \) forward under \eqref{eq:b-minus-d00-second-factor} using Lemma~\ref{lem:pushforward-sij} then gives
\begin{align*}
  u_1^0 \cdot \pi_2^* \Psi & =
  u_1^0 \cdot \sqbig{u_2^2 + u_2^1(2x_3 -2\beta_1 -2\alpha_1) + u_2^0(8\beta_2 + 2\alpha_1^2 + 4\alpha_1\beta_1 -4\alpha_1x_3 -4\beta_1x_3 + 2x_3^2) } \\
  & \mapsto
  u_1^2 + 2u_1^1(2x_2 -2\beta_1 -2\alpha_1) + 3u_1^0(8\beta_2 + 2\alpha_1^2 + 4\alpha_1\beta_1 -4\alpha_1x_2 -4\beta_1x_2 + 2x_2^2).
\end{align*}
Finally, we pull back along the projectivization by substituting \( x_2 \mapsto -\alpha_1 \), and by using \eqref{eq:uij-on-pv3-values} for the values of the \( u_1^i \):
\[
\begin{split}
  &\phantom{= .}
  3\alpha_2 + 2\alpha_1(-2\alpha_1 -2\beta_1 -2\alpha_1) + 3(8\beta_2 + 2\alpha_1^2 + 4\alpha_1\beta_1 +4\alpha_1^2 +4\alpha_1\beta_1 + 2\alpha_1^2)\\
  & =
  3\alpha_2 + 24\beta_2 + 16\alpha_1^2 + 20\alpha_1\beta_1
  .
  \end{split}
\]
This does not give us a new relation for \( \CH^*(B \sm \Delta_{00}) \), as it is implied by the relation found in the previous section.

\subsection{The Chow ring of \texorpdfstring{\boldmath \( B \sm \pi^{-1}(\Delta_{0}) \)}{B \ π⁻¹(Δ₀)}}
By Lemma~\ref{lem:loci-in-b}, $B \sm \pi^{-1}(\Delta_{0})$ is the locus of $(H,\lin_1,\lin_2)$ in which one of the $\lin_i$ divides $H$ or $H$ has a double root.
The latter locus we have already excised above.
Thus it remains to excise the former.

Arguing as in \eqref{eq:norm-pb-square} in \S\ref{subsec:b-excisions-d000}, this can be described as the norm from $\Gm \times \Gm$ to $G$ of the locus where a chosen linear factor divides the cubic.
After passing to projectivizations, the latter is the pullback under
\begin{equation}\label{eq:b-excisions-d0-id-phi}
  \id \times \phi \colon \Pj V_3 \times \Pj(V_1 \boxtimes W_{-2}) \sm I \rightarrow \Pj V_3 \times (\Pj V_1 \times \Pj V_1)
\end{equation}
of the image of
\[
  \theta \colon \Pj V_2 \times \Pj V_1 \times \Pj V_1\rightarrow \Pj V_3 \times \Pj V_1 \times \Pj V_1
\]
defined by $(f,g,h) \rightarrow (fg, g, h)$.
Arguing as in \eqref{eq:id-times-phi-pb-square} in \S\ref{subsec:b-excisions-d000}, the sought-after ideal in $\CH_{\GL_2 \times \Gm \times \Gm}^*$ is $(\id \times \phi)^*\im(\theta_*)$.
Since pullback along $\theta$ is surjective, the ideal $\im(\theta_*)$ is generated by the pushforward of the fundamental class.

The map $\theta$ factors through a diagonal map on the middle $\Pj V_1$ factor
\begin{equation}\label{eq:b-excisions-d0-diag}
  \id \times \Delta \times \id
  \colon
  \Pj V_2 \times \Pj V_1 \times \Pj V_1
  \to
  \Pj V_2 \times (\Pj V_1 \times \Pj V_1) \times \Pj V_1
\end{equation}
followed by the multiplication map
\begin{equation}\label{eq:b-excisions-d0-mult}
  \mul \times \id \times \id \colon
  (\Pj V_2 \times \Pj V_1) \times \Pj V_1 \times \Pj V_1 \rightarrow \Pj V_3 \times \Pj V_1 \times \Pj V_1.
\end{equation}
The pushforward of the fundamental class along \eqref{eq:b-excisions-d0-diag} is $[\Delta_{23}] = x_2 + x_3 - \alpha_1$ by Lemma~\ref{lem:diagonal-class}.
To pushforward along \eqref{eq:b-excisions-d0-mult}, we use Lemma \ref{lem:pushforward-sij}:
\[
  [\Delta_{23}] = u_1^0 \cdot \pbig{u_2^1 + u_2^0 \cdot (x_3 - \alpha_1)}
  \mapsto u_1^1 + 3 u_1^0(x_3 - \alpha_1).
\]
To pull back along \eqref{eq:b-excisions-d0-id-phi} and pull back from the projectivization, we substitute \( x_3 \mapsto - \alpha_1 - 2t_2 \) by Lemmas~\ref{lem:hyperplane-t2-twist}~and~\ref{lem:proj-pb-formula}, and use Lemma~\ref{lem:uij-formulas} to plug in for the \( u_1^i \).
This yields
\[
  u_1^1 + 3 u_1^0(x_3 - \alpha_1) \mapsto
  \alpha_1 + 3(-\alpha_1 - 2t_2 - \alpha_1) = -5\alpha_1 -6t_2.
\]
Pushing this and its product with \( t_1 \) forward along the norm map from  \( \Gm \times \Gm \) to \( \B G \) using Lemma~\ref{lem:norm-map-and-presentations} as in \eqref{eq:B-minus-d000-norm-push}, we get
\begin{align*}
  -5\alpha_1 -6t_2
  & \mapsto
  -10\alpha_1 - 6(\beta_1 +\gamma) = -10\alpha_1 - 6\beta_1 \\
  t_1(-5\alpha_1 -6t_2)
  & \mapsto
  -5\alpha_1(\beta_1 +\gamma) - 6(2\beta_2) = -5\alpha_1\beta_1 +\alpha_1\gamma - 12\beta_2.
\end{align*}
Using Lemma~\ref{lem:alpha-to-lambda-sub}, these give the two relations for \( \CH^*\pbig{ B \sm \pi\I(\Delta_0) } \) in Proposition~\ref{propn:biell-summary}.

\section{Excising \texorpdfstring{$\Delta_0 \sm \Delta_{01}$}{Δ₀ \ Δ₀₁}}\label{sec:excising-delta0}
As explained in the introduction, having excised \( \Delta_{001} \) and \( \Delta_{01} \sm \Delta_{001} \) from \( \Mb_2 \), we will now excise \( \Delta_{0^r} \sm (\Delta_{0^{r+1}} \cup \Delta_{0^r 1}) \), for \( r = 3,2,1 \), where \( 0^r \) denotes a string of \( r \) zeroes, and where by convention, we set \( \Delta_{0000} = \Delta_{0001} = \emptyset \).

The starting point of our computation is the computation in \cite[Lemmas~3.3,~3.9]{vistoli-chow-ring-of-m2} of the images of the first maps in the excision sequences
\begin{equation}\label{eq:vistolis-sequence}
  \CH^*\big(\Delta_{0^r} \sm (\Delta_{0^{r+1}} \cup \Delta_{0^r 1})\big)
  \to
  \CH^*\pbig{\Mb_2 \sm (\Delta_{0^{r+1}} \cup \Delta_1)}
  \to
  \CH^*\pbig{\Mb_2 \sm (\Delta_{0^r} \cup \Delta_1)}
  \to
  0,
\end{equation}
which for $r = 1$ yields $\CH^*(\M_2)$.
This computation proceeds entirely inside of $\Mb_2 \sm \Delta_1$, which admits a nice quotient presentation, which we will recall shortly.
Our task will then be to ``lift'' the result of these computations to \( \Mb_2 \sm \Delta_{01} \), where such a quotient presentation is not available.
This will allow us to compute the corresponding image in the excision sequences
\begin{equation}\label{eq:our-sequence}
  \CH^*\pbig{\Delta_{0^r} \sm (\Delta_{0^{r+1}} \cup \Delta_{0^r 1})}
  \to
  \CH^*\pbig{\Mb_2 \sm (\Delta_{0^{r+1}} \cup \Delta_{01})}
  \to
  \CH^*\pbig{\Mb_2 \sm (\Delta_{0^r} \cup \Delta_{01})}
  \to
  0,
\end{equation}
which for $r = 1$ yields $\CH^*(\M_2^\ct)$.

We recall from \cite{vistoli-chow-ring-of-m2} or \cite{larson-chow-m2} that \( \Mb_2 \sm \Delta_1 \) admits a presentation as an open substack of \( V_6(2) / \GL_2 \) by tracking the branch points of the hyperelliptic map \( C \to \Pj^1 \) of a genus 2 curve \( C \).
Specifically, \( \Mb_2 \sm \Delta_1 \subset V_6(2) / \GL_2 \) is the complement of the locus of polynomials with triple roots.

With respect to this presentation, the substack \( \Delta_{0^r} \sm (\Delta_{0^{r+1}} \cup \Delta_{0^r 1}) \subset \Mb_2 \sm (\Delta_{0^{r+1}} \cup \Delta_1)\) for \( r = 1,2,3 \) is given by the locus of polynomials with exactly $r$ double roots.
Note that this is the pre-image under $\pi_\Pj \colon \Mb_2 \sm (\Delta_{0^{r+1}} \cup \Delta_1) \to \Pj V_6$ of the image of the squaring and multiplication map
\[
  \pi_r = \mul \circ (\sq \times \id) \colon \Pj V_r \times \Pj V_{6 - 2r}  \to \Pj V_6.
\]
We thus have a pullback square
\begin{equation}\label{eq:pv6-pb-square}
  \begin{tikzcd}
    \Delta_{0^r} \sm (\Delta_{0^{r+1}} \cup \Delta_{0^r 1}) \ar[d, "", hookrightarrow] \ar[r, "\pi_{\Pj}"] \pb &
    (\Pj V_r \times \Pj V_{6 - 2r}) / \GL_2
    \ar[d, "\pi_r"] \\
    \Mb_2 \sm (\Delta_{0^{r+1}} \cup \Delta_{1}) \ar[r, "\pi_{\Pj}"] &
    \Pj V_6 / \GL_2.
  \end{tikzcd}
\end{equation}
We will be focused on the following particular classes:
\begin{defn}\label{defn:beta-alpha-ri}
  Let \( \xi_r \in \CH^*_{\GL_2}(\Pj V_r \times \Pj V_{6 - 2r}) \) denote the pullback along the first projection map of the hyperplane class on \( \Pj V_r \).
  For \( 1 \le r \le 3 \) and \( 0 \le i \le r \), we define
  \[
    \beta_{ri} \in \CH^*\big(\Mb_2 \sm (\Delta_{0^{r + 1}} \cup \Delta_{0^r 1})\big)
  \]
  to be the pushforward of \( \pi_\Pj^* (\xi_r)^i \in \CH^*\pbig{\Delta_{0^r} \sm (\Delta_{0^{r+1}} \cup \Delta_{0^r 1})} \) along the inclusion
  \[\Delta_{0^r} \sm (\Delta_{0^{r+1}} \cup \Delta_{0^r 1}) \hto \Mb_2 \sm (\Delta_{0^{r + 1}} \cup \Delta_{0^r 1}).\]
  We also write
  \begin{align*}
    \alpha_{ri} &\in \CH^*\big(\Mb_2 \sm (\Delta_{0^{r + 1}} \cup \Delta_1)\big) \\
    \bar{\beta}_{ri} &\in \CH^*\big(\Mb_2 \sm (\Delta_{0^{r + 1}} \cup \Delta_{01})\big)
  \end{align*}
  for the restrictions of \( \beta_{ri} \) to \( \Mb_2 \sm (\Delta_{0^{r+1}} \cup \Delta_1) \) and \(\CH^*\big(\Mb_2 \sm (\Delta_{0^{r + 1}} \cup \Delta_{01})\big)\) respectively.

  Note that since \( \xi_r^0 = 1 \) for each \( r \), it follows that \( \beta_{r0} \) is simply the fundamental class of \( \Delta_{0^r} \sm (\Delta_{0^{r+1}} \cup \Delta_{0^r1}) \) in \( \Mb_2 \sm (\Delta_{0^{r+1}} \cup \Delta_{0^r1}) \).
\end{defn}

\noindent
Our definition of the \(\alpha_{ri}\) is compatible with the one in \cite{vistoli-chow-ring-of-m2}, so we have:
\begin{lem}[{\cite[Lemmas~3.3,~3.9]{vistoli-chow-ring-of-m2}}]\label{lem:vistolis-formulas}
  The classes \(\alpha_{ri} \in \CH^*\big(\Mb_2 \sm (\Delta_{0^{r + 1}} \cup \Delta_1)\big)\) generate the image of \( \CH^*\pbig{\Delta_{0^r} \sm (\Delta_{0^{r+1}} \cup \Delta_{0^r 1})} \to \CH^*\pbig{\Mb_2 \sm (\Delta_{0^{r+1}} \cup \Delta_1)} \) as an ideal, and they are given by the following formulas:
  \begin{align*}
    \alpha_{30} &= - 24\lambda_1^3 + 128\lambda_1\lambda_2 = 0 &
    \alpha_{20} &= -12\lambda_1^2+144\lambda_2 &
    \alpha_{10} &= 10\lambda_1 \\
    \alpha_{31} &= 24\lambda_1^4-128\lambda_1^2\lambda_2 = 0 &
    \alpha_{21} &= 24\lambda_1^3 - 168\lambda_1\lambda_2 = 0&
    \alpha_{11} &= 2\lambda_1^2-24\lambda_2 \\
    \alpha_{32} &= - 24\lambda_1^5 + 128\lambda_1^3\lambda_2 = 0 &
    \alpha_{22} &= -24\lambda_1^4+148\lambda_1^2\lambda_2 = 0 \\
    \alpha_{33} &= 24\lambda_1^6-128\lambda_1^4\lambda_2 = 0. &
  \end{align*}
\end{lem}
Thus, using the sequences \eqref{eq:vistolis-sequence}, the rings \( \CH^*\pbig{\Mb_2 \sm (\Delta_{0^r} \cup \Delta_1)} \) are completely determined by the \( \alpha_{ri} \).
Similarly, by Lemma~\ref{lem:beta_ri-setup} below, and the sequences \eqref{eq:our-sequence}, the rings \( \CH^*\big(\Mb_2 \sm (\Delta_{0^r} \cup \Delta_{01})\big)\) are completely determined by the \( \beta_{ri} \); see \S\ref{subsec:conclusion}.

\begin{lem}\label{lem:beta_ri-setup}
  The images of
  \( \CH^*\pbig{\Delta_{0^r} \sm (\Delta_{0^{r+1}} \cup \Delta_{0^r1})} \to \CH^*\pbig{\Mb_2 \sm (\Delta_{0^{r+1}} \cup \Delta_{0^r1})} \)
  respectively
  \( \CH^*\pbig{\Delta_{0^r} \sm (\Delta_{0^{r+1}} \cup \Delta_{0^r1})} \to \CH^*\pbig{\Mb_2 \sm (\Delta_{0^{r+1}} \cup \Delta_{01})} \)
  are generated as an ideal by the \( \beta_{ri} \) respectively the \(\bar{\beta}_{ri}\).
  Moreover, the classes \( \beta_{ri} \) satisfy the following conditions.
  \begin{enumerate}[(i)]
  \item\label{item:beta_ri-setup-alpha_ri} Each \( \beta_{ri} \) maps to \( \alpha_{ri} \) under the restriction map
  \[\CH^*\pbig{\Mb_2 \sm (\Delta_{0^{r+1}} \cup \Delta_{0^r 1})} \to \CH^*\pbig{\Mb_2 \sm (\Delta_{0^{r+1}} \cup \Delta_{1})}.\]
  \item\label{item:beta_ri-setup-disjoint} Each \( \beta_{ri} \) maps to \( 0 \) under the restriction map
  \[
  \CH^*\pbig{\Mb_2 \sm (\Delta_{0^{r+1}} \cup \Delta_{0^r 1})}
  \to \CH^*(\Delta_1 \sm \Delta_{0^r1}).
  \]
  \item\label{item:beta_ri-setup-biell} Each \( \beta_{ri} \) maps to \( 0 \) under pullback
  \[
  \pi^* \colon
  \CH^*\pbig{\Mb_2 \sm (\Delta_{0^{r+1}} \cup \Delta_{0^r 1})}
  \to \CH^*(B \sm \pi\I(\Delta_{0^r}))
  \]
  along the map \( \pi \colon B \to \Mb_2 \) constructed in \S\ref{sec:bielliptic}.
  \end{enumerate}
\end{lem}

\begin{proof}
  By the projective bundle formula, the $\xi_r^i \in \CH^*_{\GL_2}(\Pj V_r \times \Pj V_{6 - 2r})$ for $i = 0, \ldots , r$ generate $\CH^*_{\GL_2}(\Pj V_r \times \Pj V_{6 - 2r})$ as a $\CH^*_{\GL_2}(\Pj V_6)$-module.
  Using \eqref{eq:pv6-pb-square}, we find $\CH^*\pbig{\Delta_{0^r} \sm (\Delta_{0^{r+1}} \cup \Delta_{0^r 1})}$ is generated as a $\CH^*\big(\Mb_2 \sm (\Delta_{0^{r+1}} \cup \Delta_{0^r 1})\big)$-module by their pullbacks $\pi_\Pj^*\xi_r^i$.
  This proves the \(\beta_{ri}\) and \(\bar{\beta}_{ri}\) generate the claimed images.

  Property \ref{item:beta_ri-setup-alpha_ri} is simply the definition of the \( \alpha_{ri} \).
  Properties \ref{item:beta_ri-setup-disjoint}~and~\ref{item:beta_ri-setup-biell} hold since \( \beta_{ri} \) already vanishes on \( \Mb_2 \sm \Delta_{0^r} \), and the maps from \( \Delta_1 \sm \Delta_{0^r1} \) and \( B \sm \pi\I(\Delta_{0^r})\) to \( \Mb_2 \sm (\Delta_{0^{r+1}} \cup \Delta_{0^r 1}) \) both factor through \( \Mb_2 \sm \Delta_{0^r} \).
\end{proof}

\noindent
It now remains to determine the \( \beta_{ri} \).
The rest of this section is devoted to proving the following:
\begin{propn}\label{propn:betaij-values}
  The classes \( \beta_{ri} \) are given as follows:
  \begin{align*}
    \beta_{30} &=
    96 \delta_1 \lambda_2
    &
    \beta_{20} &=
    60 \lambda_1^2 - 12 \delta_1 \lambda_1
    &
    \beta_{10} &=
    10 \lambda_1 - 2 \delta_1
    \\
    \beta_{31} &= 0
    &
    \beta_{21} &=
    24 w_{21} \delta_1 \lambda_2
    &
    \beta_{11} &=
    2\lambda_1^2 - 24\lambda_2 + \delta_1 \lambda_1 - \delta_1^2
    \\
    \beta_{32} &=
    24 w_{32} \delta_1 \lambda_2^2
    &
    \beta_{22} &= 0
    \\
    \beta_{33} &=
    0
  \end{align*}
  for some \(w_{21}, w_{32} \in \set{0,1} \).
\end{propn}
Note that as \( \beta_{10} \) is the fundamental class of \( \Delta_{0} \) inside of \( \Mb_2 \sm (\Delta_{00} \sm \Delta_{01}) \), its value is already given by Theorem~\ref{thm:chow-M2bar}.

\subsection{\boldmath Computations of the \texorpdfstring{$\beta_{3i}$}{β₃ᵢ}}\label{subsec:alpha3i-comp}
We recall from Theorem~\ref{thm:chow-M2bar} the presentation:
\begin{equation}\label{eq:mb2-pres-repeated}
  \CH^*(\Mb_2)\cong
  \Z[\lambda_1, \lambda_2, \delta_1 ]/
  (
  24\lambda_1^2 - 48\lambda_2,
  20\lambda_1\lambda_2 - 4\delta_1\lambda_2,
  \delta_1^3 + \delta_1^2 \lambda_1,
  2\delta_1^2 + 2\delta_1 \lambda_1
  ).
\end{equation}
Referring to \eqref{eq:ch-of-delta1} in \S\ref{subsec:d1-pres}, we find that
\[
  \ker\pbig{\CH^*(\Mb_2) \to \CH^*(\Delta_1)} =
  \pbig{
    2(\delta_1+\lambda_1),
    \delta_1(\delta_1+\lambda_1)
  }
  \subset
  \CH^*(\Mb_2).
\]
By Proposition 4.1, we have
\[
  \begin{split}
    &\ker \pBig{\CH^*(\Mb_2) \to \CH^*\pbig{B \sm \pi\I(\Delta_{000})}}
    \\&=
    \pbig{
      48 \lambda_2,
      24 \lambda_2 + 6 \delta_1^2,
      24 \lambda_2 \delta_1,
      4 \lambda_1 \lambda_2 + 4 \lambda_2 \delta_1,
      8 \lambda_1^2 - 24 \lambda_2 - 2 \delta_1^2,
      16 \lambda_2 \delta_1 + 2 \delta_1^3
    }
    \subset \CH^*(\Mb_2).
  \end{split}
\]
Considerations in \S\ref{subsec:chow-of-m2-minus-delta01} show the kernel of \( \CH^*(\Mb_2) \to \CH^*(\Mb_2 \sm \Delta_1) \) is \( (\delta_1) \subset \CH^*(\Mb_2) \).

By Lemma~\ref{lem:beta_ri-setup}, and since each \( \alpha_{3i} = 0 \), all the \( \beta_{3i} \) must lie in the intersection of these three ideals, which is given by
\[
  (
  24 \lambda_2 \delta_1
  )
  \subset \CH^*(\Mb_2).
\]
This ideal has no non-zero elements in degrees 4 or 6, and is generated by
\(
24 \lambda_2^2 \delta_1
\)
in degree 5.
We conclude that \( \beta_{31} = \beta_{33} = 0 \), and that
\[
  \beta_{30} = 24 x \lambda_2 \delta_1
  \quad \text{and} \quad
  \beta_{32} = 24 y \lambda_2^2 \delta_1
\]
for some \( x,y \in \Z \).
Since \( \beta_{30} = 96 \lambda_2 \delta_1 \) and \( \lambda_2 \delta_1 \ne 0 \) in \( \CH^*(\Mb_2) \otimes \Q \) by \cite[\S10, p.~321]{mumford1983towards}, it follows that \( x = 4 \).
Explicit calculations of intersection products in $\CH^*(\Mb_2) \otimes \Q$ then show \( x = 4 \) \cite[\S10, p.~321]{mumford1983towards}.
We may moreover take \( y \in \set{0,1} \) since \( 48 \lambda_2^2 \delta_1 = 0 \in \CH^*(\Mb_2) \).

\subsection{\boldmath Computations of the \texorpdfstring{$\beta_{2i}$}{β₂ᵢ}}\label{subsec:alpha2i-comp}
We work in the ring \( \CH^*\pbig{\Mb_2 \sm (\Delta_{000} \cup \Delta_{001})} \), which by Proposition~\ref{propn:delta-summary} and the above computation of the \( \beta_{3i} \) is given by
\begin{alignat}{2} \label{eq:m2-d000-d001}
  \CH^*\pbig{\Mb_2 \sm (\Delta_{000} \cup \Delta_{001})} &\cong
  \mathrlap{\CH^*(\Mb_2 \sm \Delta_{001}) / (\beta_{30}, \beta_{31}, \beta_{32}, \beta_{33})}
  \\ & \cong
  \Z[\lambda_1, \lambda_2, \delta_1 ]/
  (&&
  24\lambda_1^2 - 48\lambda_2,
  20\lambda_1\lambda_2 - 4\delta_1\lambda_2,
  \\ &&& 
  \delta_1^2(\delta_1+\lambda_1),
  2\delta_1(\delta_1 + \lambda_1),
  48 \delta_1 \lambda_2,
  \beta_{32}
  ).
\end{alignat}

Referring to Proposition~\ref{propn:delta-summary}, we find that
\[
  \begin{split}
    &\ker(\CH^*\pbig{\Mb_2 \sm (\Delta_{000} \cup \Delta_{001})}
    \to \CH^*(\Delta_1 \sm \Delta_{001}))=\\
    &=\pbig{
      2(\delta_1+\lambda_1),
      \delta_1(\delta_1+\lambda_1),
      24\lambda_1\lambda_2,
      144\lambda_2
    }
    \subset
    \CH^*\pbig{\Mb_2 \sm (\Delta_{000} \cup \Delta_{001})}.
  \end{split}
\]
By Proposition~\ref{propn:biell-summary}, we have
\begin{align*}
    &\ker \pBig{
    \CH^*\pbig{\Mb_2 \sm (\Delta_{000} \cup \Delta_{001})} \to \CH^*\pbig{B \sm \pi\I(\Delta_{00})}
    }
    \\
    &=
    (
    24 \lambda_2,
    4 \lambda_1 - 2 \delta_1,
    6 \delta_1^2,
    6 \lambda_2 \delta_1,
    2 \lambda_2 \delta_1 - 2 \delta_1^3,
    )
    \subset
    \CH^*\pbig{\Mb_2 \sm (\Delta_{000} \cup \Delta_{001})}.
\end{align*}
The intersection of these two ideals is
\begin{equation}\label{eq:beta_2i_ideal_intersection}
  \pbig{
    4 \lambda_1 (\delta_1 + \lambda_1),
    4 \lambda_2 (\delta_1 + \lambda_1),
    144 \lambda_2
  } \subset
  \CH^*\pbig{\Mb_2 \sm (\Delta_{000} \cup \Delta_{001})}.
\end{equation}
The kernel of \( \CH^*\pbig{\Mb_2 \sm (\Delta_{000} \cup \Delta_{001})} \to \CH^*(\Delta_1 \sm \Delta_{001}) \) is again the principal ideal \( (\delta_1) \).
The intersection of this with \eqref{eq:beta_2i_ideal_intersection} is
\[
  (
  24 \lambda_2 \delta_1,
  72 \delta_1^2
  )
  \subset \CH^*\pbig{\Mb_2 \sm (\Delta_{000} \cup \Delta_{001})}.
\]
By Lemma~\ref{lem:beta_ri-setup} and since \( \alpha_{21} = \alpha_{22} = 0 \), the classes \( \beta_{21} \) and \( \beta_{22} \) must lie in the above intersection.
This intersection is trivial in degree \( 4 \), and is generated by
\(
24 \lambda_2 \delta_1
\)
in degree 3.
We conclude that \( \beta_{22} = 0 \) and that \( \beta_{21} = 24 x \lambda_2 \delta_1 \) for some \( x \in \Z \).  We can take \( x \in \set{0,1} \) since \( 48 \lambda_2 \delta_1 = 0 \in \CH^*\pbig{\Mb_2 \sm (\Delta_{000} \cup \Delta_{001})} \).

As for \( \beta_{20} \), since it lies in \eqref{eq:beta_2i_ideal_intersection} and has degree 2, it must be equal to
\( 4 y \lambda_1 (\delta_1 + \lambda_1) + 144 z \lambda_2 \)
for some \( y,z \in \Z \).
Since \( \beta_{20} = -12 \delta_1 \lambda_1 + 60 \lambda_1^2 \) rationally by \cite[\S10, p.~321]{mumford1983towards}, it follows that \( y = -3 \) and \( z = 1 \).

\subsection{\boldmath Computation of \texorpdfstring{$\beta_{11}$}{β₁₁}}\label{subsec:alpha1i-comp}
We work in the ring
\begin{alignat}{2} \label{eq:m2-d00-d01}
  \CH^*\pbig{\Mb_2 \sm (\Delta_{00} \cup \Delta_{01})} & \cong
  \mathrlap{\CH^*(\Mb_2 \sm \Delta_{01}) / (\beta_{30}, \beta_{31}, \beta_{32},\beta_{33},\beta_{20}, \beta_{21}, \beta_{22})}
  \\ & \cong
  \Z[\lambda_1, \lambda_2, \delta_1 ]/
  (&&
  24\lambda_1^2 - 48\lambda_2,
  20\lambda_1\lambda_2 - 4\delta_1\lambda_2,
  \delta_1^3 + \delta_1^2 \lambda_1,
  \\ &&& 
  2\delta_1^2 + 2\delta_1 \lambda_1,
  12\delta_1\lambda_1,
  24\delta_1\lambda_2,
  60 \lambda_1^2 - 12 \delta_1 \lambda_1
  )
\end{alignat}

Referring to Proposition~\ref{propn:delta-summary}, we have
\[
\begin{split}
  &\ker(\CH^*(\Mb_2 \sm (\Delta_{00} \cup \Delta_{01})) \to
  \CH^*(\Delta_1 \sm \Delta_{01}))\\
  & =
  \pbig{
    2(\delta_1+\lambda_1),
    \delta_1(\delta_1+\lambda_1),
    12\delta_1,
    24\lambda_2
  }
  \subset
  \CH^*(\Mb_2 \sm \Delta_{01}).
\end{split}
\]
Thus, the general form of \( \beta_{11} \) is
\[
  \beta_{11}=
  (2a\lambda_1+b\delta_1)(\delta_1+\lambda_1)+
  24c\lambda_2,
\]
where we can assume $b\in\set{0,-1}$.

Using $\beta_{11} \mapsto \alpha_{11} \in \CH^*\pbig{\Mb_2 \sm (\Delta_{00} \cup \Delta_1)}$, we obtain
\[
  2a\lambda_1^2+24c\lambda_2 =
  2\lambda_1^2-24\lambda_2\in\CH^*(\Mb_2 \sm (\Delta_{00} \cup \Delta_{01}))/(\delta_1)
\]
and hence $a=1+12x$ and $c=-1-2x$ for some $x\in\Z$.
As usual, since $24\lambda_1^2-48\lambda_2=0\in\CH^*(\Mb_2 \sm \Delta_{01})$, we may take $x=0$.

Thus, we have
\[
  \beta_{11}=
  2\lambda_1^2-24\lambda_2+
  (2+b)\delta_1\lambda_1+
  b\delta_1^2.
\]
Unfortunately, the condition $\beta_{11}\mapsto 0\in\CH^*(B \sm \pi\I(\Delta_0))$ does not help in determining \( b \).
Instead, we show directly that \(\beta_{11}\) vanishes in \( \M_2^\ct \) for \(b = -1\).
This implies that \( b = -1 \), as we will now show.
Indeed, if \( b = 0 \), it would follow that the value of \(\beta_{11}\) for \(b = -1\) lies in the ideal of $\CH^*\pbig{\Mb_2 \sm (\Delta_{00} \cup \Delta_{01})}$ generated by \( \beta_{10} = 10\lambda_1 - 2\delta_1 \) and the value of \( \beta_{11} \) for \( b = 0 \), which one can easily check is not the case.

To show directly that \( \beta_{11} \) vanishes in \( \M_2^\ct \) when \(b = -1\), we may first suppose the characteristic is zero.
Indeed, since $\M_2^\ct$ is smooth over $\operatorname{Spec}\Z$, we have by \cite[Corollary~20.3]{fulton-intersection-theory} a specialization map: a ring homomorphism $\CH^*\pbig{(\M_2^\ct)_{\Q}} \to \CH^*\pbig{(\M_2^\ct)_{\F_p}}$ from the Chow ring of the general fibre to that of a special fibre ($p\ne2,3$).

Consider the universal curve $\pi \colon \cC \to \M_2^\ct$ and the universal separating node $\cS_1\subset\cC$.
In characteristic zero, \cite[\S6.2]{larson-chow-m2} shows that $\pi_*\pbig{\ch_1(\omega_\pi)\cdot[\cS_1]}=12\lambda_1^2-24\lambda_2$.
\begin{lem}
    The line bundle $\omega_\pi|_{\cS_1}$ is isomorphic to the pullback of the representation $\Gamma$ from $\B G$.
\end{lem}
\begin{proof}
    Recall the presentation $\Delta_{1} \cong \pbig{(L_{4,6} \sm 0)^2}/G$
  from \S\ref{subsec:quotient-presentations}.  Pulling back to the double cover $\pbig{(L_{4,6} \sm 0)^2}/(\Gm \times \Gm)$, the residue map along a chosen factor $(L_{4,6} \sm 0)/ \Gm \cong \Mb_{1,1}$ gives a trivialization of $\omega_\pi|_{\cS_1}$.  As the residue on the other factor differs by a sign, we moreover obtain a $G$-equivariant isomorphism of $\omega_\pi|_{\cS_1}$ with the pullback of $\Gamma$ from $\B G$.
\end{proof}
Thus, $\pi_*\pbig{\ch_1(\omega_\pi)\cdot[\cS_1]}=12\lambda_1^2-24\lambda_2$ is the pushforward from \(\Delta_1 \sm \Delta_{01}\) of \(\gamma\), which is \(\delta_1^2 + \delta_1\lambda_1\).
We therefore conclude
\[
  12\lambda_1^2-24\lambda_2
  =
  \delta_1^2+\delta_1\lambda_1
  \qquad
  \text{and hence}
  \qquad
  2\lambda_1^2-24\lambda_2
  +\delta_1\lambda_1
  -\delta_1^2
  =0
\]
in $\CH^*(\M_2^\ct)$, as required, since $10\lambda_1^2 - 2\delta_1\lambda_1 = \lambda_1\beta_{10} = 0\in\CH^*(\M_2^\ct)$.

\subsection{Conclusion}\label{subsec:conclusion}
We now put together Theorem~\ref{thm:chow-M2bar}, Lemma~\ref{lem:beta_ri-setup} and Proposition~\ref{propn:betaij-values} to compute the Chow rings of \( \Mb_2 \sm \Delta_{000} \), \( \Mb_2 \sm \Delta_{00} \), and \( \M_2^\ct \).
We also recall from Proposition~\ref{propn:delta-summary} and \eqref{eq:m2-d000-d001},~\eqref{eq:m2-d00-d01} the Chow rings of \( \Mb_2 \sm \Delta_{001} \), \( \Mb_2 \sm \Delta_{01} \), \( \Mb_2 \sm (\Delta_{000} \cup \Delta_{001}) \), and \( \Mb_2 \sm (\Delta_{00} \cup \Delta_{01}) \), and using Lemma~\ref{lem:beta_ri-setup} and Proposition~\ref{propn:betaij-values} again, we further compute those of \( \Mb_2 \sm (\Delta_{000} \cup \Delta_{01}) \), \( \Mb_2 \sm (\Delta_{000} \cup \Delta_{1}) \), and \( \Mb_2 \sm (\Delta_{00} \cup \Delta_{1}) \).
Finally, for completeness, we also recall the Chow ring of \( \CH^*(\M_2) \) from \cite{vistoli-chow-ring-of-m2}.

We note that the only Chow rings that are not determined completely are those of \( \Mb_2 \sm \Delta_{000} \), \( \Mb_2 \sm \Delta_{00} \), and \( \Mb_2 \sm (\Delta_{000} \cup \Delta_{001}) \), and these are each determined up to two or three possibilities.

\begin{thm}\label{thm:main-thm}
  Over any field of characteristic \( \ne 2,3 \), we have, {\color{red!60!black} for some \( w_{21},w_{32} \in \set{0,1} \)}:
  \[
    \begin{split}
      \CH^*\pbig{\Mb_2} &\cong
      \Z[\lambda_1, \lambda_2,\delta_1] / ({
        24\lambda_1^2 - 48\lambda_2,
        20\lambda_1\lambda_2 - 4\delta_1\lambda_2,
        \delta_1^3 + \delta_1^2 \lambda_1,
        2\delta_1^2 + 2\delta_1 \lambda_1
      })
      \\
      {\color{red!60!black} \CH^*(\Mb_2 \sm \Delta_{000})}
      & \cong
      \CH^*(\Mb_2) / ({
        96 \delta_1 \lambda_2,
        24 w_{32} \delta_1 \lambda_2^2
      })
      \\
      & \cong
      \Z[\lambda_1, \lambda_2, \delta_1] / ({
        24 \lambda_1^2 - 48 \lambda_2,
        20 \lambda_1 \lambda_2 - 4 \delta_1 \lambda_2,
        \delta_1^3 + \delta_1^2 \lambda_1,
        2 \delta_1^2 + 2 \delta_1 \lambda_1,
      } \\ & \hspace{7em} { 
        96 \delta_1 \lambda_2,
        24 w_{32} \delta_1 \lambda_2^2
      }),
      \\
      {\color{red!60!black} \CH^*(\Mb_2 \sm \Delta_{00})}
      & \cong
      \CH^*\pbig{\Mb_2 \sm (\Delta_{000} \cup \Delta_{001})} / ({
        60 \lambda_1^2 - 12 \delta_1 \lambda_1,
        24 w_{21} \delta_1 \lambda_2
      })
      \\
      & \cong
      \Z[\lambda_1, \lambda_2, \delta_1] / ({
        24 \lambda_1^2 - 48 \lambda_2,
        20 \lambda_1 \lambda_2 - 4 \delta_1 \lambda_2,
        \delta_1^2 (\delta_1 + \lambda_1),
        2 \delta_1 (\delta_1 + \lambda_1),
      } \\ & \hspace{7em} { 
        48 \delta_1 \lambda_2,
        24 w_{32} \delta_1 \lambda_2^2,
        60 \lambda_1^2 - 12 \delta_1 \lambda_1,
        24 w_{21} \delta_1 \lambda_2
      }),
      \\
      \CH^*(\M_2^\ct) &\cong
      \CH^*\pbig{\Mb_2 \sm (\Delta_{00} \cup \Delta_{01})} / ({
        10 \lambda_1 - 2 \delta_1,
        2 \lambda_1^2 - 24 \lambda_2 + \delta_1 \lambda_1 - \delta_1^2
      })
      \\
      & \cong
      \Z[\lambda_1, \lambda_2, \delta_1] / ({
        12 \delta_1 \lambda_1,
        24 \delta_1 \lambda_2,
        10 \lambda_1 - 2 \delta_1,
        2 \lambda_1^2 - 24 \lambda_2 + \delta_1 \lambda_1 - \delta_1^2
      })
      \\
      \CH^*(\Mb_2 \sm \Delta_{001}) &\cong
      \Z[\lambda_1, \lambda_2, \delta_1 ] / ({
        24 \lambda_1^2 - 48 \lambda_2,
        20 \lambda_1 \lambda_2 - 4 \delta_1 \lambda_2,
        \delta_1^3 + \delta_1^2 \lambda_1,
      } \\ & \hspace{7em} { 
        2 \delta_1^2 + 2 \delta_1 \lambda_1,
        144 \delta_1 \lambda_2
      })
      \\
      {\color{red!60!black} \CH^*\pbig{\Mb_2 \sm (\Delta_{000} \cup \Delta_{001})}}
      &\cong
      \Z[\lambda_1, \lambda_2, \delta_1 ] / ({
        24 \lambda_1^2 - 48 \lambda_2,
        20 \lambda_1 \lambda_2 - 4 \delta_1 \lambda_2,
      } \\ & \hspace{7em} { 
        \delta_1^2 (\delta_1 + \lambda_1),
        2 \delta_1 (\delta_1 + \lambda_1),
        48 \delta_1 \lambda_2,
        24 w_{32} \delta_1 \lambda_2^2
      })
      \\
      \CH^*(\Mb_2 \sm \Delta_{01}) &\cong
      \Z[\lambda_1, \lambda_2, \delta_1 ] / ({
        24 \lambda_1^2 - 48 \lambda_2,
        20 \lambda_1 \lambda_2 - 4 \delta_1 \lambda_2,
        \delta_1^3 + \delta_1^2 \lambda_1,
      } \\ & \hspace{7em} { 
        2 \delta_1^2 + 2 \delta_1 \lambda_1,
        12 \delta_1 \lambda_1,
        24 \delta_1 \lambda_2
      })
      \\
      \CH^*\pbig{\Mb_2 \sm (\Delta_{000} \cup \Delta_{01})} &\cong
      \CH^*(\Mb_2 \sm \Delta_{01}) / (\beta_{30}, \beta_{31}, \beta_{32}, \beta_{33}) \cong
      \CH^*(\Mb_2 \sm \Delta_{01})
      \\
      \CH^*\pbig{\Mb_2 \sm (\Delta_{00} \cup \Delta_{01})} &\cong
      \Z[\lambda_1, \lambda_2, \delta_1 ]/
      ({
        24 \lambda_1^2 - 48 \lambda_2,
        20 \lambda_1 \lambda_2 - 4 \delta_1 \lambda_2,
        \delta_1^3 + \delta_1^2 \lambda_1,
      } \\ & \hspace{7em} { 
        2 \delta_1^2 + 2 \delta_1 \lambda_1,
        12 \delta_1 \lambda_1,
        24 \delta_1 \lambda_2,
        60 \lambda_1^2 - 12 \delta_1 \lambda_1
      })
      \\
      \CH^*\pbig{\Mb_2 \sm \Delta_{1}} &\cong
      \CH^*\pbig{\Mb_2} / (\delta_1)
      \\ & \cong
      \Z[\lambda_1,\lambda_2] / ({
        24 \lambda_1^2 - 48 \lambda_2,
        20 \lambda_1 \lambda_2
      })
      \\
      \CH^*\pbig{\Mb_2 \sm (\Delta_{000} \cup \Delta_{1})} &\cong
      \CH^*(\Mb_2 \sm \Delta_{1}) / (\beta_{30}, \beta_{31}, \beta_{32}, \beta_{33},) \cong
      \CH^*(\Mb_2 \sm \Delta_{1})
      \\
      \CH^*\pbig{\Mb_2 \sm (\Delta_{00} \cup \Delta_{1})} &\cong
      \CH^*(\Mb_2 \sm \Delta_{1}) / (\beta_{30}, \beta_{31}, \beta_{32}, \beta_{33}, \beta_{20}, \beta_{21}, \beta_{22})
      \\ & \cong
      \Z[\lambda_1, \lambda_2] / ({
        24 \lambda_1^2 - 48 \lambda_2,
        20 \lambda_1 \lambda_2,
        60 \lambda_1^2
      })
      \\
      \CH^*\pbig{\M_2} &\cong
      \Z[\lambda_1, \lambda_2] / ({
        10 \lambda_1,
        2 \lambda_1^2 - 24 \lambda_2
      }).
    \end{split}
  \]
\end{thm}

\printbibliography
\end{document}